\documentclass[11pt]{amsart}
\usepackage[sorted-cites]{amsrefs}
\usepackage{verbatim}
\usepackage{color}
\usepackage{amsmath,amsthm,amscd}
\usepackage{calrsfs}
\DeclareMathAlphabet{\pazocal}{OMS}{zplm}{m}{n}
\usepackage{lipsum}

\usepackage{xcolor}
\usepackage{ mathrsfs }

\usepackage{amssymb}
\usepackage{comment}
\usepackage{enumerate}
\usepackage{graphicx}
\usepackage{hyperref}
\usepackage[english]{babel}

\definecolor{ao(english)}{rgb}{0.0, 0.5, 0.0}
\definecolor{antiquefuchsia}{rgb}{0.57, 0.36, 0.51}
\definecolor{awesome}{rgb}{1.0, 0.13, 0.32}
\newtheorem{theorem}{Theorem}[section]

\newtheorem{lemma}{Lemma}[section]
\newtheorem{cor}{Corollary}[section]
\newtheorem{prop}{Proposition}[section]

\newtheorem{definition}{Definition}[section]

\newtheorem{example}{Example}[section]
\theoremstyle{remark}
\newtheorem{remark}{Remark}[section]
\numberwithin{equation}{section}

\newcommand{\R}{\mathbb{R}}

\renewcommand{\bar}{\overline}

\begin{document}

\title[Volume properties and rigidity on self-expanders]{Volume properties and rigidity on self-expanders   of mean curvature flow }

\date{}


\thanks{The first author is supported by CAPES of Brasil.  The second author  is partially supported by CNPq and Faperj of Brazil.}

\address{Instituto de Matem\'atica e Estat\'\i stica, Universidade Federal Fluminense,
Niter\'oi, RJ 24020, Brazil}

\author[Saul Ancari]{saul Ancari}

\email{sa\_ancari@id.uff.br}

\author[Xu Cheng]{Xu Cheng}

\email{xucheng@id.uff.br}

\newcommand{\M}{\mathcal M}

\begin{abstract} In this paper, we mainly study   immersed self-expander hypersurfaces  in   Euclidean space whose mean curvatures have some linear growth controls. We discuss  the volume growths and the finiteness  of  the weighted volumes. We prove some theorems that characterize the hyperplanes through the origin as self-expanders.  We  estimate 
upper bound of the bottom of the spectrum of the drifted Laplacian. We also  give the  upper and lower bounds for  the bottom  of the spectrum of the $L$-stability operator and discuss the $L$-stability of some special self-expanders. Besides, we prove that the surfaces  $\Gamma\times\mathbb{R}$  with the product metric are the  only  complete self-expander surfaces immersed  in $\mathbb{R}^3$ with constant scalar curvature,  where $\Gamma$ is a complete self-expander curve (properly) immersed in $\mathbb{R}^2$.

\end{abstract}

\maketitle
\section{introduction}\label{introduction}

In this article we  study self-expanders that are self-expanding solutions for mean curvature flows.
An $n$-dimensional smooth  immersed  submanifold  $\Sigma$  in the Euclidean space $\R^{m}$ is called self-expander if   its  mean curvature vector $ {\bf H}$  satisfies the equation
\begin{equation}{\bf H}=\dfrac12x^{\perp},
\end{equation}
where $x^{\perp}$ denotes the normal component of the position vector $x$.\\

Equivalently, $\Sigma$ is a self-expander if and only if $\sqrt{t}\Sigma, t\in (0,\infty)$ is a mean curvature flow (MCF).\\

Self-expanders have a very  important role  in the study of  MCF. They describe the asymptotic longtime behavior for MCF  and the local structure of MCF after the singularities in the very short time.  In \cite{EH}, Ecker and Huisken  studied MCF evolutions of entire graphical immersions. Under some assumptions on the initial hypersurface at infinity,  they showed that the solution of MCF  exists for all times $t>0$ and  converges to a self-expander.  Stavrou \cite{S} later proved the same result under weaker hypotheses that the initial hypersurface  has a unique tangent cone at infinity. Self-expanders also appears in  the mean curvature evolution of cones.  In \cite{I}, Ilmanen studied the existence of E-minimizing self-expanding hypersurfaces  which converge to prescribed closed cones at infinity. In \cite{D},  Ding  studied self-expanders and their relationship to minimal cones in Euclidean space.    Recently, Bernstein and Wang (\cite{BW1} and \cite{BW2}) obtained various results on asymptotically conical self-expanders. There are other works in self-expanders. See, for instance, \cite{AIC}, \cite{FM}, \cite{smoczyk2020self}, etc.\\
 
Recently in \cite{cheng2018}, the second author of the present paper and Zhou studied  some properties of complete properly immersed  self-expanders. Especially, they proved the discreteness of the spectrum of the drifted operator $\mathcal{L}=\Delta+\dfrac12\langle x,\nabla \cdot\rangle $. In the  case of self-expander hypersurfaces,    they gave the lower bound estimate for the  first eigenvalue  $\lambda_1$  of the operator $\mathcal{L}$   and also  proved that the bottom $\mu_1$ of the spectrum of the   stability operator $L=\mathcal{L}+|A|^2-\frac12$ satisfies $\mu_1\leq \lambda_1+\frac12$ (\cite[Theorems 1.3, 1.5]{cheng2018}). Besides, they proved the uniqueness of hyperplanes through the origin for mean convex self-expanders under some integrability conditions on the square of the norm of the second fundamental form (\cite[Theorems 1.4]{cheng2018}).\\

 Motivated by the work in \cite{cheng2018}, in the present  paper we  study the topics discussed in
  \cite{cheng2018}. One of our  strategies is  to make use of  the properties on the finiteness of  weighted volumes and the volume growth upper estimate for self-expanders  with some restriction on mean curvature. In order to prove these properties, 
we first prove   Theorem \ref{tstimat} on   a  general Riemannian manifold which generalizes Theorem 5 in \cite{AM} and is also of independent interest.
Then  we apply Theorem \ref{tstimat} to  self-expanders  in $\mathbb{R}^m$ and  obtain the following result: 

\begin{theorem}\label{ineq5}
	Let $\Sigma$ be a complete $n$-dimensional properly immersed   self-expander   in $\R^m$, $n<m$.  Assume  that  its mean curvature vector $\bf{H}$ satisfies  $|{\bf H}|(x)\leq a|x|+b$, $x\in \Sigma$, for  some constants $0\leq a<\frac{1}{2}$ and  $b>0$. Then it holds that, for any  $\alpha >\frac{4a^2}{1-4a^2},$
	\begin{itemize}
		\item[(i)] $\int_{\Sigma}e^{-\frac{\alpha}{4}|x|^2}d\sigma<\infty.$
		\item[(ii)] The volume of $B_r(0)\cap\Sigma$ satisfies 
		\begin{align*}
		 Vol(B_r(0)\cap\Sigma)\leq C(\alpha)e^{\frac{\alpha }{4}r^2},
		\end{align*}
	\end{itemize}
	 where $B_r(0)$ denotes  the  round ball in $\mathbb{R}^m$ of radius $r$ centered at the origin  $0\in \mathbb{R}^m$. 
	 
	In particular,  if  $0\leq a<\frac{1}{2\sqrt{2}}$,  then  the Gaussian weighted volume is finite, that is,
	\begin{align}\label{ps}
	\int_{\Sigma}e^{-\frac{|x|^2}{4}}d\sigma<\infty.
	\end{align} 
	
\end{theorem}

\begin{remark} 
	The partial  conclusion in Theorem \ref{ineq5} ($\int_{\Sigma}e^{-\frac{|x|^2}{4}}d\sigma<\infty$ if $0\leq a<\frac1{2\sqrt{2}}$) can also be proved using Theorem 5 in \cite{AM}.  It is noted    that  the restriction  $a_2<\frac{1}{4}$ should be added in the assumption of Theorem 5 in \cite{AM}.   
\end{remark}
A natural question arises,  whether  \eqref{ps} holds for any complete properly immersed self-expanders.  In this direction, we have the following result:
\begin{theorem}\label{ineq4}
		Let $\Sigma$ be a complete $n$-dimensional properly immersed   self-expander   in $\R^m$, $n<m$.   If there are some  constants  $0\leq a<\frac12$, $b\geq 0$ and $r_0>0$ such that  the mean curvature vector of $\Sigma$ satisfies that  $| { \bf H}|(x)\geq a|x|+b$ for   $x\in\Sigma\setminus B_{r_0}(0)$, then 
	\begin{align*}
	\int_{\Sigma}e^{-\frac{4a^2}{1-4a^2}|x|^2}d\sigma=\infty.
	\end{align*} 
		In particular, if there are some  constants $b\geq 0$ and $r_0>0$ such that  the mean curvature vector of $\Sigma$ satisfies that  $| { \bf H}|(x)\geq\frac{1}{2\sqrt{2}}|x|+b$ for   $x\in\Sigma\setminus B_{r_0}(0)$,  then 
	\begin{align*}
	\int_{\Sigma}e^{-\frac{|x|^2}{4}}d\sigma=\infty.
	\end{align*}
\end{theorem}

 In this paper,  the notation $A$  denotes  the second fundamental form of $\Sigma$. We obtain the following rigidity property of hyperplanes as  self-expanders.
   \begin{theorem}\label{se33}
  	Let $\Sigma$ be a complete properly immersed self-expander hypersurface in $\R^{n+1}$.   Assume  that  its  mean curvature $H$ satisfies $|H|(x)\leq a|x|+b$, $x\in\Sigma$, for  some constants $0\leq a<\frac{1}{2}$ and   $b>0$.  If there exists $\beta>0$, such that 
  	$$|A|^2H^2 +\frac{1}{2}H^2+\beta A(x^T,x^T)H\leq0,$$
  	then $\Sigma$ must be a hyperplane $\R^{n}$ through the origin,  where $x^T$ denotes the
  	tangent component of the position vector $x$.
  \end{theorem} 

Theorem \ref{se33} is a consequence of a more general result (see Theorem \ref{se3}). \\

We also prove the following result:

\begin{theorem}\label{se334}
	Let $\Sigma$ be a complete properly immersed  self-expander hypersurface in $\R^{n+1}$. Assume that its mean curvature $H$ is bounded from below and satisfies  $H(x)\leq a|x|+b$, $x\in\Sigma$, for  some constants $0\leq a<\frac{1}{2}$ and $b>0$. If   there exists  $\alpha>\frac{4a^2}{1-4a^2}$ such that 
	$$|A|^2H+\frac{H}{2} +\frac{\alpha+1}{4}A(x^T,x^T)\geq0,$$
	then  $\Sigma$ must be a  hyperplane  $\R^{n}$ through the origin, where $x^T$ denotes the
	tangent component of the position vector $x$.
\end{theorem}

 In Section \ref{results} of this paper, we study the problems related to the spectrum of the drifted Laplacian  $\mathcal{L}$. In \cite{cheng2018},  the second author of the present paper and Zhou (\cite[Theorems 1.1 and 1.3]{cheng2018}) proved that the spectrum  of the  drifted Laplacian $\mathcal{L}$ on a properly immersed $n$-dimensional self-expander $\R^{n+k}$, $k\geq1$, is discrete. In particular, the bottom $\lambda_1$ of the spectrum of $\mathcal{L}$ is the first weighted $L^2$ eigenvalue of $\mathcal{L}$. Further, for codimension $1$ case, they  proved that  $\displaystyle \lambda_1\geq\frac{n}{2} +\inf_{x\in\Sigma}H^2$  and  this lower bound  is achieved if  and only if the self-expander is the hyperplane  through the origin. 
In this paper we give an upper bound
 for the bottom of the spectrum of the drifted Laplacian $\mathcal{L}$ and  discuss the rigidity of the upper bound. More precisely, we prove that
 
  \begin{theorem} \label{thm-3-hyper} Let $\Sigma$ be  a complete properly immersed self-expander hypersurface in $\mathbb{R}^{n+1}$. Assume that its mean curvature $H$ satisfies $|H|(x)\leq a|x|+b$, $x\in\Sigma$, for some constants $0\leq a<\frac{1}{2\sqrt{2}}$ and $b>0$.  Then the bottom $\lambda_1$ of the spectrum of the drifted Laplacian $\mathscr{L}=\Delta+\frac{1}{2}\left<x, \nabla\cdot\right>$ on $\Sigma$, i.e. the first weighted $L^2$ eigenvalue of $\mathscr{L}$ satisfies
 	\begin{align}
 	\lambda_1
 	&\leq \frac{n}{2}+\frac{\int_{\Sigma}H^2 e^{-\frac{|x|^2}{4}}d\sigma}{\int_{\Sigma}e^{-\frac{|x|^2}{4}}d\sigma},
 	\end{align}
 	with  equality  if and only if $\Sigma$ is the hyperplane $\mathbb{R}^n$ through the origin.
 \end{theorem}

 \begin{remark}
 If $0\leq a<\frac{1}{2}$, we have a general upper bound estimate for $\lambda_1$. See Theorem \ref{thm-3-hyper-1}.
 \end{remark}

  In this paper, we also study the $L$-stability operator for self-expanders: 
  \begin{align}
  L=\mathcal{L}+|A|^2-\frac{1}{2}.
  \end{align}
    It is well known that a self-expander is noncompact (see, e.g., \cite{CH}). Thus the bottom $\mu_1$ of the spectrum of the operator $L$ may take $-\infty$ and  for $\mu_1>-\infty$, it  may not be the lowest weighted $L^2$-eigenvalue for $L$.
  If $\mu_1\geq 0$, $\Sigma$ is called $L$-stable.
 $L$-stability means that the second variation of its weighted volume is nonnegative for any compactly supported normal variation.     In this paper,  we obtain a lower bound for $\mu_1$ as follows: 
  \begin{theorem}\label{staab}
  	Let $\Sigma$ be a complete immersed self-expander hypersurface in $\R^{n+1}$. Then 
  	\begin{align}\label{bee}
  	\mu_1\geq\frac{n+1}{2}+\inf_{x\in\Sigma}Scal_{\Sigma},
  	\end{align}
  	where $\inf_{x\in\Sigma}Scal_{\Sigma}$ denotes the infimum  of the scalar curvature $Scal_{\Sigma}=H^2-|A|^2$ of $\Sigma$.
  	
  	Moreover,  the equality in \eqref{bee} holds  if  $\Sigma$ is also properly immersed, has   constant scalar curvature and satisfies   $|H|(x)\leq a|x|+b$, $x\in\Sigma$, for some constants $0\leq a<\frac{1}{2\sqrt{2}}$ and $b>0$.
  \end{theorem}
 In \cite{cheng2018}, the second author of the present paper and Zhou proved that the mean convex self-expanders are $L$-stable. Here  Theorem \ref{staab}  implies that
  \begin{cor}\label{cor-stab} Let $\Sigma$ be a complete immersed self-expander hypersurface in $\mathbb{R}^{n+1}$. If the scalar curvature of  $\Sigma$ satisfies 
  	\begin{align*}
  	Scal_{\Sigma}\geq -\frac{n+1}{2},
  	\end{align*}
  	then $\Sigma$ is $L$-stable.
  	\end{cor}

 The authors in \cite{cheng2018} also proved that the inequality  $\mu_1\leq \lambda_1+\frac12$ holds on complete properly immersed self-expander hypersurfaces. In this paper, we obtain  another  upper bound for $\mu_1$. More precisely,
\begin{theorem} \label{staab2} Let $\Sigma$ be  a complete properly immersed self-expander hypersurface in $\mathbb{R}^{n+1}$.   Assume  that its mean curvature $H$ satisfies   $|H|(x)\leq a|x|+b$,  $x\in\Sigma$, for some constants  $0\leq a<\frac{1}{2\sqrt{2}}$ and $b>0$.   Then the bottom $\mu_1$ of the spectrum of the $L$-stability operator $L$  satisfies
	\begin{align}\label{bee1}
	\mu_1
	&\leq \frac{n+1}{2}+\frac{\int_{\Sigma}Scal_{\Sigma} e^{-\frac{|x|^2}{4}}d\sigma}{\int_{\Sigma}e^{-\frac{|x|^2}{4}}d\sigma}.
	\end{align}
	If $\mu_1>-\infty$, then $\int_{\Sigma}Scal_{\Sigma}e^{-\frac{|x|^2}{4}}d\sigma<\infty$ and the equality \eqref{bee1}
	holds if and only if the curvature $Scal_{\Sigma}$ of $\Sigma$ is constant.
\end{theorem}

 \begin{remark} In the more general case of $0\leq a<\frac{1}{2}$, we also obtain  a general upper bound estimate for $\mu_1$. See Theorem \ref{staab2-1}.
 \end{remark}

 Theorems  \ref{staab} and \ref{staab2} have the following consequence:
\begin{cor}\label{corr1}
	Let $\Gamma$ be  a complete  immersed self-expander curve in $\R^2$ and $\Sigma$ be  the self-expander hypersurface $\Gamma\times\R^{n-1}$ with the product metric, where $n\geq 1$. Then the bottom $\mu_1$ of the spectrum of the $L$-stability operator on $\Sigma$ is $\frac{n+1}{2}$. In particular, the bottom  of the spectrum of the $L$-stability operator of $\Gamma$ is $1$.
\end{cor}

\begin{remark}
	Note that in Corollary \ref{corr1}   we do  not assume any hypothesis on the mean curvature. 
\end{remark}

 Noting that   the equality  in \eqref{bee1}  holds  if and only if the scalar curvature $ Scal_\Sigma $ is constant, we are
interested in characterizing self-expander hypersurfaces with constant scalar curvature, which is also of independent interest.  In this direction, we prove  Theorem \ref{scc} which states that  $\Gamma\times\mathbb{R}^{n-1}$  with the product metric are the  only  complete  self-expander hypersurfaces immersed in $\mathbb{R}^{n+1}$ with nonnegative constant scalar curvature,  where $\Gamma$ is a complete immersed self-expander curve in $\R^2$.   Theorem \ref{scc} is a  consequence of Proposition \ref{pmf} which states that a complete self-expander hypersurface immersed $\mathbb{R}^{n+1}$ different from a hyperplane and with nonnegative scalar curvature is of the form  $\Gamma\times\mathbb{R}^{n-1}$, where $\Gamma$ is a complete non-trivial self-expander  curve  immersed in $\mathbb{R}^2$, if and only if the scalar curvature attains a local minimum  on the open set $\{x\in\Sigma; H(x)\neq0\}$. In its proof, we use a result by Smoczyk which gave the equivalent property of the self-expander hypersurfaces $\Gamma\times\mathbb{R}^{n-1}$ (Theorem 5.1 in \cite{smoczyk2020self}). In general, it would be interesting to ask if the following is true.\\

\textbf{Problem:} Let $\Sigma$ be a complete immersed self-expander hypersurface in $\mathbb{R}^{n+1}$ with constant scalar curvature. Is it true that   $\Sigma=\Gamma\times\mathbb{R}^{n-1}$ with the product metric, where $\Gamma$ is a complete self-expander curve immersed in $\mathbb{R}^2$?\\

For self-expander surfaces in $\mathbb{R}^3$, we obtain the following result:
\begin{theorem}\label{scc6}
	Let $\Sigma$ be a complete  immersed self-expander surface in $\mathbb{R}^3$. If the scalar curvature of $\Sigma$ is constant, then  $\Sigma=\Gamma\times\mathbb{R}$  with the product metric, where $\Gamma$ is a complete self-expander curve (properly) immersed in $\mathbb{R}^2$.
\end{theorem}

 Self-expander curves in $\mathbb{R}^2$ have been classified  (see the work of Ishimura \cite{Is} and Halldorsson \cite{Hall}). Theorem 6.1 in \cite{Hall} states that each of the complete self-expander curves immersed in $\mathbb{R}^2$  is convex, properly embedded and asymptotic to
the boundary of a cone with vertex at the origin. It is the graph of an even function.
The curves form a one-dimensional family parametrized by their distance to the
origin, which can take on any value in $[0, \infty)$.   \\

Theorems   \ref{staab2} and   \ref{scc6} have the following  consequence. 

\begin{cor}
	Let $\Sigma$ be a complete properly immersed self-expander surface  in $\R^{3}$. Assume that  $|H|(x)\leq a|x|+b$, $x\in\Sigma$, for some constants $0\leq a<\frac{1}{2\sqrt{2}}$ and $b>0$. Then 
	\begin{align}\label{spn66}
		\mu_1\leq\frac{3}{2}+\frac{\int_{\Sigma}Scal_{\Sigma} e^{-\frac{|x|^2}{4}}d\sigma}{\int_{\Sigma}e^{-\frac{|x|^2}{4}}d\sigma}.
	\end{align}
	If $\mu_1>-\infty$, then
	the equality \eqref{spn66} holds if and
	only if  $\Sigma=\Gamma\times\R$  with the product metric,  where $\Gamma$ is a  complete self-expander curve (properly) immersed in $\R^2$.
\end{cor}
For self-expander surfaces in $\mathbb{R}^3$, we also obtain the following result:
\begin{theorem}\label{sfc}
	Let $\Sigma$ be a complete properly immersed self-expander surface in $\mathbb{R}^3$. If  $\Sigma$ has nonpositive scalar curvature and the norm of its second fundamental form  is constant, then  $\Sigma$ is a plane $\mathbb{R}^2$ through the origin.
\end{theorem}

The rest of the paper is organized as follows:  In Section \ref{notation}, we recall some notations and  basic facts. In Section \ref{finiteness}, we prove Theorem \ref{tstimat} and apply it to prove  the finiteness of weighted volumes and volume growth estimate of   self-expanders, that is, Theorem \ref{ineq5}. We also prove Theorem \ref{ineq4}.  In Section \ref{hyperplane}, we prove rigidity Theorems. In Section \ref{results}, we obtain an upper bound for the bottom of the spectrum  of the drifted Laplacian $\mathcal{L}$.  In Section \ref{scalar}, we discuss self-expanders with constant scalar curvature and prove Theorem  \ref{scc6}. We also prove Theorem \ref{sfc}. In Section \ref{section-L}, we  obtain the  upper and  lower bounds for the bottom of the spectrum of the $L$-stability operator  $L$ and also discuss the $L$-stability of  self-expanders.

\section{Preliminaries}\label{notation}

In this section, we will recall some concepts and basic facts.\\

Assume that    $\left(M,\bar{g}\right)$ is a smooth  $m$-dimensional  Riemannian manifold.   Let $\Sigma$ be an $n$-dimensional immersed submanifold in $M$ with the induced metric $g$.
We will denote by  $d\sigma$ the volume form  of $\Sigma$. In this paper, unless otherwise specified, the notations with a bar,  for instance $\bar{\nabla}$ and  $\bar{\nabla}^2$,  denote the quantities
corresponding the metric $\bar{g}$ on $M$.  On the other hand, the notations  like  $\nabla, \Delta$ denote the quantities corresponding the intrinsic metric ${g}$ on $\Sigma$.\\

 The  isometric immersion $i: (\Sigma^n, g) \to (M^{n+k},\bar{g})$ is said to be properly immersed if,  for any compact subset $\Omega$ in $M$, the pre-image $i^{-1}(\Omega)$ is compact in $\Sigma$.\\




 Let $A$ denote the second fundamental form of $(\Sigma,g)$, that is,  at $p\in \Sigma$, 
$A(X,Y)= (\overline{\nabla}_XY)^\perp,$
where $X,Y\in T_p\Sigma$,  $\perp$ denotes the projection onto the normal bundle of $\Sigma$. 
The mean curvature vector ${\bf H}$ of $\Sigma$  is defined as 
 the trace of  $A$. If $\Sigma$ is a  hypersurface, its  mean curvature $H$   is defined by 
${\bf H}=-H{\bf n},$
where ${\bf n}$ is the unit normal field on $\Sigma$.  \\

Given a smooth function $f$ on $M$, define the weighted mean curvature vector ${\bf H}_f$ of a submanifold $(\Sigma, g)$ by ${\bf H}_f:={\bf H}+(\overline{\nabla }f)^{\perp} $.   $\Sigma$  is called $f$-minimal  if  its weighted mean curvature vector  ${\bf H}_f$ vanishes identically, or equivalently if it satisfies\begin{equation} \label{f-min}{\bf H}=-(\overline{\nabla} f)^\perp.
\end{equation}
If $\Sigma$ is a hypersurface, its   weighted  mean curvature  $H_f$ is defined by  $ {\bf H}_f =-H_f{\bf n}$. In particular $\Sigma$ is $f$-minimal if and only if the weighted mean curvature satisfies that
$H_f=H-\left<\bar{\nabla}f, {\bf n}\right>=0$ or equivalently
\begin{equation}
	H=\left<\bar{\nabla}f, {\bf n}\right>.
\end{equation}

The weighted volume of a measurable subset  $S\subset \Sigma$ with respect  to  the function $f$ is defined by
\begin{equation}\label{notation-eq-vol}V_f(S):=\int_S e^{-f}d\sigma.
\end{equation}

It is known that  an $f$-minimal submanifold  is a critical point of the weighted volume functional defined in (\ref{notation-eq-vol}). On the other hand,  it is also  a minimal submanifold under the conformal metric $\tilde{g}=e^{-\frac{2}{n}f}\bar{g}$ on $M$  (see, e.g. \cite{CMZ3}, \cite{CMZ}).
When $(M, \bar{g})$ is the Euclidean space $(\mathbb{R}^{n+k}, g_0)$, there are  very interesting examples of $f$-minimal submanifolds:
\begin{example}\label{example1} If $f=\frac{|x|^2}4$,  $-\frac{|x|^2}4$, and  $-\left<x,w\right>$ respectively,  where $w\in \mathbb{R}^{n+k}$ is a constant vector, an $n$-dimensional $f$-minimal submanifold $\Sigma$ is a  self-shrinker, self-expander and translator for MCF in the Euclidean space $\mathbb{R}^{n+k}$ respectively. 
\end{example}

 On the smooth metric measure space $(\Sigma, g, e^{-f})$, there is a very important second-order elliptic operator: the drifted Laplacian
$
\Delta_{f}=\Delta-\left\langle \nabla f,\nabla\cdot\right\rangle .
$
It is well known that $\Delta_f$ is a densely defined self-adjoint
operator in $L^2(\Sigma, e^{-f}d\sigma)$, i.e. for 
$u$ and $v$ in $C^{\infty}_0(\Sigma)$, it holds that 
\begin{equation}
\int_{\Sigma}(\Delta_{f}u) ve^{-f}d\sigma=-\int_{\Sigma}\left\langle \nabla u,\nabla v\right\rangle e^{-f}d\sigma.
\end{equation}

 Since we will study the spectrum problems in this paper, we recall some facts in spectral theory (see more details in,  e.g. \cite{G},  \cite{RS}). Consider the Schr$\ddot{\text{o}}$dinger operator on $\Sigma$:
$$S=\Delta_f+q, \quad q\in L^{\infty}_{loc}(\Sigma).$$
The  weighted $L^2$ spectrum of $S$  is called the spectrum of $S$ for short whenever there is no confusion.
The bottom $s_1$ of the spectrum of $S$ can be characterized by
\begin{align}\label{bottom-0}
\displaystyle s_1=\inf\left\{\frac{\int_{\Sigma}\left(|\nabla\varphi|^{2}-q\varphi^2\right)e^{-f}d\sigma}{\int_{\Sigma}\varphi^2e^{-f}d\sigma};  \varphi\in C_0^{\infty}(\Sigma),  \int_{\Sigma}\varphi^2e^{-f}d\sigma\neq 0\right\}.
\end{align}
In general, if $\Sigma$ is noncompact, the bottom $s_1$ may not be  the weighted $L^2$ eigenvalue and may take $-\infty$. \\

 Now we  give especial notations on self-expanders.  In the following, unless otherwise specified, let $\Sigma$ be an $n$-dimensional  self-expander in $\mathbb{R}^{n+k}, k\geq 1$, that is, $\Sigma$ satisfies the equation
\begin{align}{\bf H}=\dfrac{x^{\perp}}{2}
\end{align}

In the case of codimension $1$, the mean curvature $H$ of a self-expander $\Sigma$ satisfies that
\begin{align}H=-\frac12\langle x, {\bf n}\rangle.
\end{align}

Observe that taking  $f=-\frac{|x|^2}{4}$, a self-expander $\Sigma$ can be viewed as an $f$-minimal submanifold in $\mathbb{R}^{n+k}$ since it satisfies the equation (\ref{f-min}).

The weighted volume of a measurable subset  $S\subset \Sigma$ is given by
\begin{equation}\label{notation-expand-vol}V(S):=\int_S e^{\frac{|x|^2}{4}}d\sigma.
\end{equation}

The weighted $L^{2}$ inner product of functions $u$ and $v$ in  $L^2(\Sigma, e^{\frac{|x|^2}{4}}d\sigma)$ is defined by 
\[
\left\langle u,v\right\rangle _{L^2(\Sigma, e^{\frac{|x|^2}{4}}d\sigma)}=\int_{\Sigma}uve^{\frac{|x|^2}{4}}d\sigma.
\]

Denote by  $\mathcal{L}$ the drifted Laplacian on $\Sigma$, i.e.  $\mathcal{L}=\Delta+\frac12\left<x,\nabla\cdot\right>$.\\

The bottom $\lambda_1$ of the spectrum of $\mathcal{L}$  can be given by
\begin{align}\label{bottom-1}
\displaystyle\lambda_1=\inf\left\{\frac{\int_{\Sigma}|\nabla \varphi|^2e^{\frac{|x|^2}{4}}d\sigma}{\int_{\Sigma}\varphi^2e^{\frac{|x|^2}{4}}d\sigma};  \varphi\in C_0^{\infty}(\Sigma),  \int_{\Sigma}\varphi^2e^{\frac{|x|^2}{4}}d\sigma\neq 0\right\}.
\end{align}

From (\ref{bottom-1}), $\lambda_1$ is nonnegative.
 For self-expanders,   the stability operator for  $\Sigma$ appeared in the second variation formula of the weighted volume is a Schr$\ddot{\text{o}}$dinger operator:
\begin{equation} \label{L} L=\mathcal{L}+|A|^2-\frac12=\Delta+\frac12\left<x, \nabla\cdot\right>+|A|^2-\frac12.
\end{equation}

\begin{definition}A self-expander $\Sigma$ is said to be $L$-stable  if the following inequality holds  for all  $\varphi\in\mathit{C}^{\infty}_{0}(\Sigma)$,
\begin{equation}\label{f-stable-ine}
-\int_{\Sigma}\varphi (L\varphi)e^{\frac{|x|^2}{4}}d\sigma=\int_{\Sigma}\left(|\nabla\varphi|^{2}-\bigr(|A|^{2}-\frac12)\varphi^2\right)e^{\frac{|x|^2}{4}}d\sigma\geq 0.
\end{equation}
\end{definition}

 $L$-stability  of $\Sigma$ is equivalent to that   the second variation of its weighted volume is nonnegative for any compactly supported normal variation. Denote the bottom of the spectrum of $L$ by $\mu_1$.
$L$-stability  means  $\mu_1\geq 0$. \\

For self-expander hypersurfaces, the following equations are known (see, for instance,  \cite{cheng2018}).
\begin{prop}\label{simo}
	If $\Sigma \subset\R^{n+1}$ is a self-expander hypersurface, then
	\begin{align}\label{eH-1}
	&\mathcal{L}H=-(|A|^2+\frac{1}{2})H,\\
	&\mathscr{L}H^2=-H^2(2|A|^2+1)+2|\nabla H|^2,\\\label{Acon}
	&\mathscr{L}|A|^2=-|A|^2(2|A|^2+1)+2|\nabla A|^2,\\\label{esc}
	&\mathscr{L}(Scal_{\Sigma})=-Scal_{\Sigma}(2|A|^2+1)+2|\nabla H|^2 -2|\nabla A|^2.
	\end{align}
\end{prop}

\section{ Weighted  volumes of properly immersed  self-expanders }\label{finiteness}

In this section, motivated by Theorem 1.1  in \cite{CVZ}, Theorem 1.1 in \cite{CZ} and Theorem 5 in \cite{AM}, we  prove the following Theorem \ref{tstimat} which deals with   the growth of volumes of the level sets of adequate  functions on a general Riemannian manifold. 
Next, we prove Theorem \ref{ineq5}.  \\

Recall that a function $h$ on a  complete noncompact Riemannian manifold $M$ is said to be  proper if, for any bounded closed subset  $I\subset \mathbb{R}$, the inverse image $h^{-1}(I)$ is compact in $M$.

\begin{theorem}\label{tstimat}
	Let $(X,g)$ be a complete noncompact Riemannian manifold. Assume that $\alpha > 0$, $\beta> 0$, $a_0$, $a_1$ and $a_2$ are constants satisfying $a_2<\frac{\beta}{4}$. If $h$ is a proper nonnegative $C^2$ function on $X$ such that
	\begin{align}
	\Delta h-\alpha|\nabla h|^2+\beta h\leq a_2r^2+a_1r+a_0
	\end{align} 
	and
	\begin{align}
	\Delta h\leq a_2r^2+a_1r+a_0
	\end{align}
	hold on the sets $D_r=\{x\in X; 2\sqrt{h}\leq r\}$ for all $r>0$, then 
	\begin{itemize}
		\item[(i)] The integral 
		\begin{align}\label{dxm}
		\int_Xe^{-\alpha h}dv<\infty.
		\end{align}
		\item[(ii)] For all $r>0$ the volume of the set $D_r$ satisfies
		\begin{align}
		V(r)\leq Ce^{\varepsilon(a_2r^2+a_1r+a_0)+\frac{\alpha r^2}{4e^{\varepsilon \gamma}}}, \mbox{  for any } \varepsilon>0,
		\end{align}
		and 
		\begin{align}
		V(r)\leq Ce^{\frac{\alpha }{4}r^2},
		\end{align}
		where $\gamma=\frac{\beta}{\alpha}$ and $C=C(\alpha)=\int_Xe^{-\alpha h}dv<\infty.$
	\end{itemize}
\end{theorem}
\begin{remark}
If  $a_2=0$ and $a_1=0$,  \eqref{dxm} was obtained in \cite{CVZ} (see (1.6) of Theorem 1.1 in \cite{CVZ}). However,  in this case, the level sets $D_r$ have polynomial volume growth as proved in \cite{CVZ}. 
\end{remark}
\begin{remark}
 Taking $\alpha=\beta=1$  in Theorem \ref{tstimat}, we obtain Theorem 5 in \cite{AM}.  It worth mentioning   that  the restriction  $a_2<\frac14$ should be added in the assumption of Theorem 5 in \cite{AM}.
\end{remark}

\begin{proof}[Proof of Theorem \ref{tstimat}]
	Let $\gamma=\frac{\beta}{\alpha}$ and $k(r)=a_2r^2+a_1r+a_0$. Since $h$ is proper, the following integral $I(t)$ is well defined.
	\begin{align*}
	I(t)=\frac{1}{t^{k(r)}}\int_{\overline{D}_r}e^{\frac{-\alpha h}{t^{\gamma}}}dv, \mbox{ }t>0.
	\end{align*}
	\begin{align}\label{stim1}
	I'(t)=t^{-k(r)-1}\int_{\overline{D}_r}e^{\frac{-\alpha h}{t^{\gamma}}}\left( \frac{\beta h}{t^\gamma}-k(r)\right)dv. 
	\end{align}
	On the other hand,  for $  t\geq 1, \gamma>0,$
	\begin{align}\nonumber
	\int_{\overline{D}_r}\mbox{div}\left( e^{\frac{-\alpha h}{t^{\gamma}}}\nabla h\right)dv&= \int_{\overline{D}_r} e^{\frac{-\alpha h}{t^{\gamma}}}\left( \Delta h-\frac{\alpha}{t^\gamma}|\nabla h|^2\right)dv\\\nonumber 
	&=\int_{\overline{D}_r} e^{\frac{-\alpha h}{t^{\gamma}}}\left( (1-\frac{1}{t^{\gamma}})\Delta h+\frac{1}{t^\gamma}(\Delta h-\alpha|\nabla h|^2)\right)dv\\\nonumber
	&\leq \int_{\overline{D}_r} e^{\frac{-\alpha h}{t^{\gamma}}}\left[(1-\frac{1}{t^\gamma})k(r)+\frac{1}{t^{\gamma}}(-\beta h+k(r))\right]dv\\\label{stim2}
	&=\int_{\overline{D}_r} e^{\frac{-\alpha h}{t^{\gamma}}}\left( k(r)-\frac{\beta h}{t^\gamma}\right)dv.  
	\end{align}
	Substituting \eqref{stim2} into \eqref{stim1} gives
	\begin{align*}
	I'(t)\leq -t^{-k(r)-1}\int_{\overline{D}_r}\mbox{div}( e^{\frac{-\alpha h}{t^{\gamma}}}\nabla h)dv.
	\end{align*}
	At the regular value $r$ of $h$ and for $t\geq 1$, by Stokes' Theorem, we have
	\begin{align*}
	I'(t)&\leq-t^{-k(r)-1}\int_{\partial D_r}\left\langle e^{\frac{-\alpha h}{t^{\gamma}}}\nabla h,\frac{\nabla h}{|\nabla h|} \right\rangle dv\\
	&=-t^{-k(r)-1}\int_{\partial D_r} e^{\frac{-\alpha h}{t^{\gamma}}}|\nabla h|dv\leq 0.
	\end{align*}
	Integrating $I'(t)$ over $t$ from $1$ to $e^{\varepsilon}>1$, where $\varepsilon>0$, we get
	\begin{align}\label{stim3}
	\frac{1}{e^{\varepsilon k(r)}}\int_{\overline{D}_r} e^{\frac{-\alpha h}{e^{\varepsilon\gamma}}}dv\leq\int_{\overline{D}_r}e^{-\alpha h}dv.
	\end{align}
	Since the integral in \eqref{stim3} is right continuous  in $r$, \eqref{stim3} holds for all $r>0.$ \\
	Note $2\sqrt{h}\leq r$ over $\overline{D}_r$. \eqref{stim3} implies that, for all $r>0$,
	\begin{align}\nonumber
	\frac{1}{e^{\varepsilon k(r)+\frac{\alpha  r^2}{4e^{\varepsilon \gamma}}}}\int_{\overline{D}_r}dv&\leq\frac{1}{e^{\varepsilon k(r)}}\int_{\overline{D}_r} e^{\frac{-\alpha h}{e^{\varepsilon\gamma}}}dv\\\label{stim4}
	&\leq\int_{\overline{D}_r}e^{-\alpha h}dv.
	\end{align}
	Note that  for $r\geq 1$ 
	\begin{align}\nonumber
	\int_{\overline{D}_r}e^{-\alpha h}dv-\int_{\overline{D}_{r-1}}e^{-\alpha h}dv&=\int_{\overline{D}_r\setminus\overline{D}_{r-1}}e^{-\alpha h}dv\\\label{stim5}
	&\leq e^{\frac{-\alpha(r-1)^2}{4}}\int_{\overline{D}_r}dv.
	\end{align}
	By \eqref{stim4}, \eqref{stim5} 
	\begin{align}\nonumber
	&\int_{\overline{D}_r}e^{-\alpha h}dv-\int_{\overline{D}_{r-1}}e^{-\alpha h}dv\\\nonumber
	&\leq  e^{\varepsilon k(r)+\frac{\alpha  r^2}{4e^{\varepsilon \gamma}}-\frac{\alpha(r-1)^2}{4}}\int_{\overline{D}_r}e^{-\alpha h}dv\\\label{stim6}
	&=e^{(\varepsilon a_2+\frac{\alpha}{4e^{\varepsilon\gamma}}-\frac{\alpha}{4})r^2+(\varepsilon a_1+\frac{\alpha}{2})r+(\varepsilon a_0-\frac{\alpha}{4})}\int_{\overline{D}_r}e^{-\alpha h}dv.
	\end{align}
	Since $a_2<\frac{\beta}{4}$, there exists a very small $\varepsilon_0=\varepsilon_0(\alpha, \beta, a_2)$ such that 
	\begin{align*}
	(\varepsilon_0a_2+\frac{\alpha}{4e^{\varepsilon_0 \gamma}}-\frac{\alpha}{4})<0.
	\end{align*}
	Then there exists $r_0\geq1$ such that for $r\geq r_0$
	\begin{align*}
	e^{(\varepsilon_0 a_2+\frac{\alpha}{4e^{\varepsilon_0\gamma}}-\frac{\alpha}{4})r^2+(\varepsilon_0 a_1+\frac{\alpha}{2})r+(\varepsilon_0 a_0-\frac{\alpha}{4})}\leq e^{-r}.
	\end{align*}
	Substituting into \eqref{stim6} gives that
	\begin{align}
	\int_{\overline{D}_r}e^{-\alpha h}dv\leq\frac{1}{1-e^{-r}}\int_{\overline{D}_{r-1}}e^{-\alpha h}dv.
	\end{align}
	Then for any positive integer $N$, we have 
	\begin{align}
	\int_{\overline{D}_{r+N}}e^{-\alpha h}dv\leq \left( \prod_{i=0}^{N}\frac{1}{1-e^{-(r+i)}}\right) \int_{\overline{D}_{r-1}}e^{-\alpha h}dv.
	\end{align}
	Noting that the infinite product $\prod\limits_{i=0}^{\infty}\left( 1-e^{-(r+i)}\right) $ converges to a positive number and letting $N$  tend to infinity, we get that $\int_X e^{-\alpha h}dv< +\infty.$\\
	Moreover by \eqref{stim4}, for all $r>0$ and $\varepsilon>0$,
	\begin{align}\nonumber
	\frac{1}{e^{\varepsilon k(r)+\frac{\alpha  r^2}{4e^{\varepsilon \gamma}}}}\int_{\overline{D}_r}dv
	\leq\int_{\overline{D}_r}e^{-\alpha h}dv\leq\int_{X}e^{-\alpha h}dv.
	\end{align}
	Hence 
	\begin{align*}
	V(r)\leq C(\alpha)e^{\varepsilon k(r)+\frac{\alpha  r^2}{4e^{\varepsilon \gamma}}}=C(\alpha)e^{\varepsilon (a_2r^2+a_1r+a_0)+\frac{\alpha  r^2}{4e^{\varepsilon \gamma}}},
	\end{align*}
	where $C(\alpha)=\int_{X}e^{-\alpha h}dv.$\\
	Since $\varepsilon>0$ is arbitrary, letting $\varepsilon\rightarrow 0$ yields that 
	\begin{equation*}
	V(r)\leq C(\alpha)e^{\frac{\alpha }{4}r^2}.
	\end{equation*}
\end{proof}

\begin{proof}[Proof of Theorem \ref{ineq5}]
	 Take $h=\frac{|x|^2}{4}$, $x\in \Sigma$. Since $\Sigma$ is properly immersed in $\mathbb{R}^m$, $h(x)$ is proper on $\Sigma$.  We have that, on  $B_r(0)\cap\Sigma$,
	\begin{align*}\nonumber
	\Delta h&=\frac{n}{2}+\langle \overline{\nabla}h, {\bf H}\rangle=\frac{n}{2}+|{\bf H}|^2\\
	&\leq a^2r^2+2abr+(b^2+\frac{n}{2}).
	\end{align*} 
	 In the above, we used the hypothesis: $|{\bf H}|(x)\leq a|x|+b$, $x\in \Sigma$.  We also have that, on  $B_r(0)\cap\Sigma$,
	\begin{align}\nonumber
	\Delta h-\alpha|\nabla h|^2+\alpha h=&\frac{n}{2}+|{\bf H}|^2+\alpha |(\overline{\nabla}h)^\perp|^2\\\nonumber
	=&\frac{n}{2}+(1+\alpha)|{\bf H}|^2\\\nonumber
	\leq& (1+\alpha)a^2r^2+2(1+\alpha)abr
	+(1+\alpha)b^2+\frac{n}{2}.
	\end{align}

	Let $a_2=(1+\alpha)a^2$, $a_1=2(1+\alpha)ab$, $a_0=(1+\alpha)b^2+\frac{n}{2}$ and  $\beta=\alpha$.   For  $\alpha >\frac{4a^2}{1-4a^2}$,  where $0\leq a<\frac12$,  it holds that $a_2<\frac{\beta }{4}$. By applying Theorem \ref{tstimat}, we obtain that $\int_{\Sigma}e^{-\frac{\alpha}{4}|x|^2}d\sigma<\infty$ for all  $\alpha >\frac{4a^2}{1-4a^2}$ and the volume of $B_r(0)\cap\Sigma$ satisfies that, for all $r>0,$
	\begin{align*}
	 Vol(B_r(0)\cap\Sigma)\leq C(\alpha)e^{\frac{\alpha }{4}r^2}.
	\end{align*}
	 In the particular case of  $a<\frac{1}{2\sqrt{2}}$, since $a<\frac{1}{2\sqrt{2}}$ implies that  $\frac{4a^2}{1-4a^2}<1$,  we may  take $\alpha=1$.
\end{proof}

	 By an argument analogous to  the ones used in the proofs of Theorem 4.1 in \cite{CZ} and  Theorem 4 in \cite{AM}, we may prove the following  result: Let   $\Sigma$ be a complete $n$-dimensional immersed self-expander in $\R^m$, $n<m$. If there exists $\alpha>0$ such that   $\int_{\Sigma}e^{-\frac{\alpha}{4}|x|^2}d\sigma<\infty$, then  $\Sigma$ is properly immersed on $\R^m$.   
 Hence Theorem \ref{ineq5} has the following consequence.
\begin{cor}
	Let $\Sigma$ be a complete $n$-dimensional immersed self-expander in $\R^m$, $n<m.$ Assume that its mean curvature vector  ${\bf H}$  satisfies $|{\bf H}|(x)\leq a|x|+b$, $x\in\Sigma$, for some constants $0\leq a<\frac{1}{2}$ and $b>0$. Then for $\alpha>\frac{4a^2}{1-4a^2}$  the following statements are equivalent:
	\begin{itemize}
		\item [(i)] $\Sigma$ is properly immersed on $\R^m$.
		\item [(ii)] There exist constants $C=C(\alpha)$, $\overline{a}_0$, $\overline{a}_1$, $\overline{a}_2$, $\overline{a}_2<\frac{\alpha}{4}$, such that 
		\begin{align*}
			V(B_r(0)\cap\Sigma)\leq C(\alpha)e^{\overline{a}_2r^2+\overline{a}_1r+\overline{a}_0}.
		\end{align*}
		\item [(iii)] $\int_{\Sigma}e^{-\frac{\alpha}{4}|x|^2}d\sigma<\infty.$
	\end{itemize}
\end{cor}
\begin{proof}[Proof of Theorem \ref{ineq4}] 
	  Let  $u=e^{-\frac{\alpha}{4}|x|^2}$. Since $\Sigma$ is properly immersed, using $\Delta |x|^2=2n+4|{\bf H}|^2$ and  Stokes' Theorem, we have
	\begin{align}
	\int_{ \big(B_t(0)\setminus B_{r_0}(0)\big)\cap \Sigma}u 
	&=\frac{1}{2n}\int_{ \big(B_t(0)\setminus B_{r_0}(0)\big)\cap \Sigma}u\Delta |x|^2\nonumber\\
	&\qquad -\frac{2}{n}\int_{  \big(B_t(0)\setminus B_{r_0}(0)\big)\cap \Sigma}u|{\bf H}|^2\nonumber\\
	&=-\frac{1}{n}\int_{\big(B_t(0)\setminus B_{r_0}(0)\big)\cap \Sigma}\langle \nabla u, x^T\rangle +\frac{1}{n}\int_{\partial B_t(0)\cap \Sigma}u|x^T|\nonumber\\
	&\qquad -\frac{1}{n}\int_{\partial B_{r_0}(0)\cap \Sigma}u|x^T|-\frac{2}{n}\int_{ \big(B_t(0)\setminus B_{r_0}(0)\big)\cap \Sigma}u|{\bf H}|^2.\label{ineq1}
	\end{align}
	On the other hand, using co-area formula we obtain
	\begin{align}\label{ineq2}
	&\frac{d}{d t}\left( \frac{1}{t^{n}}\int_{ \big(B_t(0)\setminus B_{r_0}(0)\big)\cap \Sigma}u\right)\nonumber \\
	&=-\frac{n}{t^{n+1}}\int_{ \big(B_t(0)\setminus B_{r_0}(0)\big)\cap \Sigma}u+\frac{1}{t^{n}}\int_{\partial B_t(0)\cap \Sigma}\frac{u|x|}{|x^T|}.
	\end{align} 
	Substituting  \eqref{ineq1} into  \eqref{ineq2} and noting that $|x|=t$ on $\partial B_t(0)\cap \Sigma$, we have
	\begin{align}\label{ineq2-2}
	 &\frac{d}{d t}\left( \frac{1}{t^{n}}\int_{ \big(B_t(0)\setminus B_{r_0}(0)\big)\cap \Sigma}u\right)\nonumber\\
	 =&\frac{1}{t^{n+1}}\int_{ \big(B_t(0)\setminus B_{r_0}(0)\big)\cap \Sigma}\left(2u|{\bf H}|^2+\langle \nabla u,  x^T\rangle \right)\nonumber\\
	&+\frac{1}{t^{n +1}}\left(\int_{\partial B_t(0)\cap \Sigma}u\frac{|x|^2-|x^T|^2}{|x^T|}+\int_{\partial B_{r_0}(0)\cap \Sigma}u|x^T|\right)\nonumber\\
	\geq&\frac{1}{t^{n+1}}
	\int_{ \big(B_t(0)\setminus B_{r_0}(0)\big)\cap \Sigma}\left(2u|{\bf H}|^2+\langle \nabla u, x^T\rangle \right)\nonumber\\
	=&\frac{1}{t^{n+1}}
	\int_{ \big(B_t(0)\setminus B_{r_0}(0)\big)\cap \Sigma} \left( 2|{\bf H}|^2-\frac{\alpha}{2}|x^T|^2\right)e^{-\frac{\alpha}{4}|x|^2} \nonumber\\
	=&\frac{1}{t^{n+1}}
	\int_{ \big(B_t(0)\setminus B_{r_0}(0)\big)\cap \Sigma} \left( 2(1+\alpha)|{\bf H}|^2-\frac{\alpha}{2}|x|^2\right) e^{-\frac{\alpha}{4}|x|^2}
	\end{align} 
	Choose  $\alpha=\frac{4a^2}{1-4a^2}$, where $0\leq a<\frac12$.  By the hypothesis $|{\bf H}|(x)\geq a|x|+b$   for  $|x|\geq r_0$, it follows that   
	\begin{align}\label{ineq3}
	\frac{\partial}{\partial t}\left( \frac{1}{t^{n}}\int_{\big(B_t(0)\setminus B_{r_0}(0)\big)\cap \Sigma} e^{-\frac{\alpha}{4}|x|^2}\right)&\geq0 \quad \text{for} \quad t\geq r_0.
	\end{align}
	Since $\Sigma$ is properly immersed, there exists some $t_0>r_0$ such that $$\int_{\big(B_{t_0}(0)\setminus B_{r_0}(0)\big)\cap \Sigma} e^{-\frac{\alpha}{4}|x|^2}>0.$$ 
	Integrating \eqref{ineq3} from $t_0$ to  $t>t_0$, we get 
	\begin{align*}
	\int_{\big(B_{t}(0)\setminus B_{r_0}(0)\big)\cap \Sigma} e^{-\frac{\alpha}{4}|x|^2}d\sigma\geq\frac{t^{n}}{t_0^{n}}\int_{\big(B_{t_0}(0)\setminus B_{r_0}(0)\big)\cap \Sigma} e^{-\frac{\alpha}{4}|x|^2}d\sigma.
	\end{align*}
	Letting $t\rightarrow\infty$ in the above inequality implies 
	\begin{align*}
	\int_{\Sigma} e^{-\frac{\alpha}{4}|x|^2}d\sigma=\infty.
	\end{align*}
	The particular case stated in the theorem holds by taking $a=\frac1{2\sqrt{2}}$ and $\alpha=1$.
	
\end{proof}

\section{Rigidity of hyperplanes} \label{hyperplane}

  In this section we study the rigidity property of  hyperplanes  as self-expanders. First we state the following equations:
 \begin{lemma}
 	Let $\Sigma$ be an immersed self-expander hypersurface in $\R^{n+1}$. Then for $\alpha\in \R$ it holds that
 	\begin{align}\label{spn7}
 	&\mathcal{L}_\alpha H=-\frac{1}{2}H-|A|^2H-\frac{\alpha+1}{2}\langle x,\nabla H\rangle,\\\label{spn8}
 	&\mathcal{L}_\alpha H=-\frac{1}{2}H-|A|^2H-\frac{\alpha+1}{4}A(x^T,x^T),
 	\end{align}
 	where the operator $\mathcal{L}_\alpha=\Delta-\frac{\alpha}{2}\langle x, \nabla \cdot\rangle$ and $x^T$ denotes the tangent component of $x$.
 \end{lemma}
 \begin{proof}
 	Since $\mathscr{L}H=-\frac{1}{2}H-|A|^2H$,
 	\begin{align*}
 	\mathcal{L}_\alpha H=-\frac{1}{2}H-|A|^2H-\frac{\alpha+1}{2}\langle x,\nabla H\rangle.
 	\end{align*}
 	Take a local orthonormal frame $\{e_i\}$, $i=1,...,n$ for $\Sigma$. From $H=-\frac{1}{2}\langle x,{\bf n}\rangle$,
 	\begin{align*}
 	2\nabla_{e_i}H=&-\langle \nabla_{e_i}x, {\bf n}\rangle-\langle x,\nabla_{e_i}{\bf n}\rangle\\ 
 	=&h_{ij}\langle x,e_j\rangle
 	\end{align*}
 	and hence 
 	\begin{align*}
 	\langle x,\nabla H\rangle=\langle x,e_i\rangle\nabla_{e_i}H=\frac{1}{2}h_{ij}\langle x,e_i\rangle\langle x,e_j\rangle=\frac{1}{2}A(x^T,x^T).
 	\end{align*}
 	By this and Equation \eqref{spn7}, we have that 
 	\begin{align*}
 	\mathcal{L}_\alpha H&=-\frac{1}{2}H-|A|^2H-\frac{\alpha+1}{2}\langle x,\nabla H\rangle\\
 	&=-\frac{1}{2}H-|A|^2H-\frac{\alpha+1}{4}A(x^T,x^T).
 	\end{align*}
 \end{proof}

Now, we  prove the following result:
\begin{theorem}\label{se3}
	Let $\Sigma$ be a complete  immersed self-expander hypersurface in $\R^{n+1}$. Assume that  $\delta\in\{1,3,5...\}$ and $\alpha>0$. If $\Sigma$ satisfies 	the following properties:
	\begin{itemize}
		\item [(i)] $|A|^2H^2 +\frac{1}{2}H^2+\frac{\alpha+1}{4}A(x^T,x^T)H\leq0$,
		\item [(ii)] $\frac{1}{j^2}\int_{B^{\Sigma}_{2j}(p)\setminus B^{\Sigma}_j(p)}H^{\delta+1}e^{-\alpha\frac{|x|^2}{4}}d\sigma\rightarrow 0$ when $j\rightarrow \infty$, for a fixed point $p\in \Sigma$,
	\end{itemize}
	then $\Sigma$ must be a hyperplane $\R^{n}$ through the origin,  where $x^T$ denotes the
	tangent component of the position vector $x$.
\end{theorem}
\begin{proof}
	
	Let $\varphi\in C^{\infty}_0(\Sigma)$. From \eqref{spn8}, hypothesis (i) and the value of $\delta$, we have
	\begin{align*}
	0&\leq\int_{\Sigma}\left( -\frac{1}{2}H-|A|^2H-\frac{\alpha+1}{4}A(x^T,x^T) \right) H^{\delta}\varphi^2e^{-\alpha\frac{|x|^2}{4}}\\
	&=\int_{\Sigma}H^{\delta}\varphi^2(\mathcal{L}_{\alpha} H)e^{-\alpha\frac{|x|^2}{4}}.
	\end{align*}
	Further, 
	\begin{align*}
	\int_{\Sigma}H^{\delta}\varphi^2(\mathcal{L}_{\alpha} H)e^{-\alpha\frac{|x|^2}{4}}
	 =&-2\int_{\Sigma}H^{\delta}\varphi\langle \nabla\varphi,\nabla H\rangle e^{-\alpha\frac{|x|^2}{4}}\\
	&-\delta\int_{\Sigma}\varphi^2H^{\delta-1}|\nabla H|^2e^{-\alpha\frac{|x|^2}{4}}\\
	\leq&-\frac{\delta}{2}\int_{\Sigma}\varphi^2H^{\delta -1}|\nabla H|^2e^{-\alpha\frac{|x|^2}{4}}\\
	&+\frac{2}{\delta}\int_{\Sigma}|\nabla\varphi|^2H^{\delta+1}e^{-\alpha\frac{|x|^2}{4}},
	\end{align*}
	where $\delta>0$. Therefore 
	\begin{align*}
	\frac{\delta}{2}\int_{\Sigma}\varphi^2H^{\delta -1}|\nabla H|^2e^{-\alpha\frac{|x|^2}{4}}\leq\frac{2}{\delta}\int_{\Sigma}|\nabla\varphi|^2H^{\delta+1}e^{-\alpha\frac{|x|^2}{4}}.
	\end{align*}
	
	Choose  $\varphi=\varphi_j$, where  $\varphi_j$ are the nonnegative cut-off functions satisfying that  $\varphi_j = 1$ on $B_j^{\Sigma}(p)$, $\varphi_j = 0$ on $\Sigma\setminus B_{2j}^{\Sigma}(p)$ and $|\nabla \varphi_j|\leq \frac{1}{j}$. By the monotone convergence theorem and hypothesis $(ii)$, it follows that, on $\Sigma$,
	$$H^{\delta-1}|\nabla H|^2=0,$$
	We claim that $H=0$ on $\Sigma$. In fact, if $H(p)\neq 0$ for some $p\in\Sigma$, then there exists a neighborhood $B_{\varepsilon}(p)$ such that $H\neq0$ on $B_{\varepsilon}(p)$. So $\nabla H=0$ on $B_{\varepsilon}(p)$ and hence $H=C$ on $B_{\varepsilon}(p)$.  By \eqref{eH-1} we conclude that $H\equiv 0$ on $B_{\varepsilon}(p)$ which contradicts with $H(p)=0$. The claim implies that   $\Sigma$ is the    hyperplane  $\R^{n}$ through the origin.
\end{proof}
We give the integrable property of the powers of the norm of mean curvature vector ${\bf H}$ which will be used later.
\begin{lemma}\label{spn10}
	Let $\Sigma$ be  a complete $n$-dimensional  properly immersed self-expander in $\R^m$. Assume that $|{\bf H}|\leq a|x|+b$, $x\in \Sigma,$  for some constants $0\leq a<\frac{1}{2}$ and $b>0$. Then for $\delta\geq 0$ and  $\alpha >\frac{4a^2}{1-4a^2}$, 
	\begin{align}
	\int_{\Sigma}|{\bf H}|^{\delta}e^{-\frac{\alpha}{4}|x|^2}d\sigma<\infty.
	\end{align}  
\end{lemma}
\begin{proof}
	Note that $|{\bf H}|=|\frac{1}{2}x^{\perp}|\leq\frac{1}{2}|x|.$ Hence it is easily to see that Theorem \ref{ineq5} implies the desired conclusion.
\end{proof}
A consequence of Theorem \ref{se3} is Theorem  \ref{se33}.

\begin{proof}[Proof of Theorem  \ref{se33}] 
The  hypothesis $|A|^2H^2 +\frac{1}{2}H^2+\beta A(x^T,x^T)H\leq0,$ implies that  $A(x^T,x^T)H\leq0$. Then for any  $\alpha>4\beta-1$, 
	\begin{align}\label{se333}
	|A|^2H^2 +\frac{1}{2}H^2+\frac{\alpha+1}{4} A(x^T,x^T)H\leq0,
	\end{align}
	which is  just the condition (i) of Theorem \ref{se3}. \\
	Choose   $\alpha >\max\{\frac{a^2}{\frac{1}{4}-a^2}, 4\beta-1\}$. By  Lemma \ref{spn10}, the  condition (ii) of Theorem \ref{se3} is also satisfied.
	\end{proof}

In order to prove Theorem \ref{se334}, we prove the following result.

 \begin{theorem}\label{se5}
	Let $\Sigma$ be a complete  immersed  self-expander hypersurface in $\R^{n+1}$. Assume that its mean curvature $H$ is bounded from below.  If there exists $\alpha>0$ such that  the following  conditions hold:
	
	\begin{itemize}
		\item [(i)] $|A|^2H+\frac{H}{2} +\frac{\alpha+1}{4}A(x^T,x^T)\geq0$,
		\item [(ii)] $\frac{1}{j^2}\int_{B^{\Sigma}_{2j}(p)\setminus B^{\Sigma}_j(p)}e^{-\alpha\frac{|x|^2}{4}}d\sigma\rightarrow 0$,  when $j\rightarrow\infty$, for a fixed point $p\in\Sigma$,
	\end{itemize}
	then  $\Sigma$ must be a  hyperplane  $\R^{n}$ through the origin,  where $x^T$ denotes the
	tangent component of the position vector $x$.
\end{theorem}
\begin{proof} 
	Let us fix $C=\inf_{x\in\Sigma}H$. From hypothesis (i) and \eqref{spn8} it follows that
	\begin{align}\label{si1}
	\mathcal{L}_{\alpha}(H-C)\leq0.
	\end{align}
	By  the maximum principle, either $H\equiv C$ or $H> C.$ 
	If $H\equiv C$, then $\Sigma$ is the  hyperplane through the origin. 
	If $H>C$, let us consider $u:=\log(H-C)$.  A computing yields
	\begin{align}\label{e14}
	\Delta u=-|\nabla u|^2+\frac{\Delta H}{H-C}.
	\end{align}     
	Combining   \eqref{si1} and \eqref{e14}, we get
	\begin{align}\label{e16}
	\mathcal{L}_{\alpha} u\leq-|\nabla u|^2.
	\end{align}

	Let us consider the sequence $\varphi_j$ of nonnegative cut-off function satisfying that  $\varphi_j=1$ on $B^{\Sigma}_j(p)$, $\varphi_j=0$ on $\Sigma\setminus B^{\Sigma}_{2j}(p)$ and $|\nabla \varphi_j|\leq \frac{1}{j}.$\\
	Multiplying \eqref{e16} by $\varphi^2_j$ and integrating by parts we obtain 
	\begin{align*}
	\int_{\Sigma}\varphi^2_j|\nabla u|^2e^{-\alpha\frac{|x|^2}{4}}&\leq -\int_{\Sigma}\varphi^2_j(\mathcal{L}_{\alpha} u)e^{-\alpha\frac{|x|^2}{4}}\\
	&=\int_{\Sigma}2\varphi_j\langle \nabla \varphi_j,\nabla u\rangle e^{-\alpha\frac{|x|^2}{4}}\\
	&\leq\frac{1}{2}\int_{\Sigma}\varphi^2_j|\nabla u|^2e^{-\alpha\frac{|x|^2}{4}}+2\int_{\Sigma}|\nabla\varphi_j|^2e^{-\alpha\frac{|x|^2}{4}}.
	\end{align*} 
	Therefore
	\begin{align}\label{e17}
	\int_{\Sigma}\varphi^2_j|\nabla u|^2e^{-\alpha\frac{|x|^2}{4}}\leq 4\int_{\Sigma}|\nabla\varphi_j|^2e^{-\alpha\frac{|x|^2}{4}}.
	\end{align}
	By hypothesis (ii) and the dominated convergence theorem, we obtain 
	\begin{align*}
	\int_{\Sigma}|\nabla u|^2e^{-\alpha\frac{|x|^2}{4}}=0.
	\end{align*}
	In particular $H$ must be a constant, but this contradicts  the assumption that $H>\inf_{x\in\Sigma}H$.
\end{proof}

 As a consequence, Theorem \ref{se5} implies Theorem \ref{se334}. 
\begin{proof}[Proof of Theorem \ref{se334}]  From the assumption on the  mean curvature of $\Sigma$, it follows that $|H|(x)\leq a|x|+b_1$, for  some constants $0\leq a<\frac{1}{2}$ and $b_1>0$. Therefore,
	by Lemma \ref{spn10}, the condition (ii) of Theorem \ref{se5} is satisfied.
\end{proof}

 Theorem  \ref{se334} also has the following consequence:
\begin{cor}\label{se334-1}
	Let $\Sigma$ be a complete properly immersed  self-expander hypersurface in $\R^{n+1}$.  Assume that its mean curvature $H$ satisfies  $H(x)\leq a|x|+b$, $x\in \Sigma,$  for  some constants $0\leq a<\frac{1}{2}$ and $b>0$. If  $A(x^T,x^T)$ is bounded  from above and there exists  $\alpha>\frac{4a^2}{1-4a^2}$ such that
	\begin{align}\label{ineq-hyo}|A|^2H+\frac{H}{2} +\frac{\alpha+1}{4}A(x^T,x^T)\geq0,
	\end{align}
	then  $\Sigma$ must be a  hyperplane  $\R^{n}$ through the origin, where $x^T$ denotes the
	tangent component of the position vector $x$.
\end{cor}
\begin{proof} We claim that $\inf_{x\in\Sigma}H>-\infty.$ In fact, if $\inf_{x\in\Sigma}H=-\infty$, then there exists a sequence $\{p_k\}$ in $\Sigma$ such that $H(p_k)\rightarrow -\infty$ when $k\rightarrow \infty$.
 By \eqref{ineq-hyo},  we have
	\begin{align}\label{me}
	|A|^2(p_k)\leq-\frac{1}{2}-\frac{(\alpha+1)A(p_k^T,p^T_k)}{4H(p_k)}.
	\end{align}
	 By the  hypothesis  that $A(x^T,x^T)$ is bounded from above,  \eqref{me} implies that $|A|^2(p_k)<0$ for $k$  large enough, that is a contradiction. Therefore  $\inf_{x\in\Sigma}H>-\infty.$\\
	
Now applying Theorem  \ref{se334}, we complete the proof.
\end{proof}

\bigskip

\section{Upper bound of $\lambda_1$}\label{results}

 In this section, we prove the upper bound estimate of the first eigenvalue $\lambda_1$ for the drifted Laplacian $\mathscr{L}.$

Let $\Sigma$ be a complete  $n$-dimensional immersed self-expander in $\R^m$ (not necessarily hypersurface). Then, the   functions 
$v=e^{-\frac{\alpha+1}{8}|x|^2}$, $\alpha\in \mathbb{R}$, satisfy
\begin{align}\label{espn1}
\mathscr{L}v+\left( \frac{(\alpha+1) n}{4}+\frac{\alpha+1}{2} |{\bf H}|^2-\frac{(\alpha+1) (\alpha-1)}{16}|x^T|^2\right)v=0\hspace{0.2cm}\mbox{on}\hspace{0.2cm}\Sigma. 
\end{align}

In fact,  it was proved in \cite[Lemma 3.1]{cheng2018} that for any smooth functions $u$, $f$ and $h$, it holds that
$$\Delta_f(ue^h)=e^h\{\Delta_{f-2h}u+[\Delta h+\langle \nabla(h-f), \nabla h\rangle]u\}.$$
Substituting  $u=1$, $f=-\frac{|x|^2}{4}$ and $h=-\frac{\alpha+1}{8}|x|^2$ into the above equality yields \eqref{espn1}.\\

  The following integrability properties on $v$ hold.

\begin{lemma}\label{spn9}
	Let $\Sigma$ be  a complete properly $n$-dimensional immersed self-expander  in $\mathbb{R}^{m}$.  Assume  that its mean curvature vector ${\bf H}$  satisfies $|{\bf H}|(x)\leq a|x|+b$, $x\in\Sigma$, for some constants $0\leq a<\frac{1}{2}$ and $b>0$. Then for   $\alpha>\frac{4a^2}{1-4a^2}$, $v=e^{-\frac{\alpha+1}{8}|x|^2}$ satisfies 
	\begin{align}
	\int_{\Sigma}v^2e^{\frac{|x|^2}{4}}d\sigma<\infty\hspace{0.5cm}\mbox{and}\hspace{0.5cm}\int_{\Sigma}|\nabla v|^2e^{\frac{|x|^2}{4}}d\sigma<\infty.
	\end{align}
\end{lemma}
\begin{proof}
	Note that 
	\begin{align*}
	v^2e^{\frac{|x|^2}{4}}= e^{-\frac{\alpha}{4}|x|^2}
	\end{align*}
	and 
	\begin{align*}
	|\nabla v|^2e^{\frac{|x|^2}{4}}=\frac{(\alpha+1)^2}{16}|x^T|^2 e^{-\frac{\alpha}{4}|x|^2}.
	\end{align*}
	It is easily to see that Theorem \eqref{ineq5} implies the desired conclusion.
\end{proof}
Now we give an inequality on the first eigenvalue $\lambda_1$.

 \begin{theorem} \label{thm-3-hyper-1} Let $\Sigma$ be  a complete properly immersed self-expander hypersurface in $\mathbb{R}^{n+1}$. Assume that its mean curvature $H$   satisfies $|H|(x)\leq a|x|+b$,  $x\in\Sigma$, for some constants $0\leq a<\frac{1}{2}$ and $b>0$. Then the bottom $\lambda_1$ of the spectrum of the drifted Laplacian $\mathscr{L}=\Delta+\frac{1}{2}\left<x, \nabla\cdot\right>$ on $\Sigma$, i.e. the first weighted $L^2$ eigenvalue of $\mathscr{L}$ satisfies
 	
 	\begin{align}\label{3-eq-thm2-1-c}
 	\lambda_1
 	&\leq \frac{(\alpha+1) n}{4}+\frac{(\alpha+1)\int_{\Sigma}\left( \frac12H^2-\frac{(\alpha-1)}{16}|x^T|^2\right) e^{-\frac{\alpha}{4}|x|^2}d\sigma}{\int_{\Sigma}e^{-\frac{\alpha}{4}|x|^2}d\sigma},
 	\end{align}
 	for all $\alpha>\frac{a^2}{\frac{1}{4}-a^2}$.\\
 	\end{theorem}

\begin{proof}
	By Theorem 1.1 in \cite{cheng2018}, the spectrum of the operator $\mathscr{L}$  is discrete and $\lambda_1$ is the first weighted $L^2$ eigenvalue of $\mathscr{L}$. Then for any   $\phi\in C^{\infty}_0(\Sigma)$
	
	\begin{align}\label{eq-bottom-2-1}
	\lambda_1\int_{\Sigma}\phi^2e^{\frac{|x|^2}{4}}
	&\leq\int_{\Sigma}|\nabla \phi|^2e^{\frac{|x|^2}{4}}.
	\end{align}
	Take $v=e^{-\frac{\alpha+1}{8}|x|^2}$, where $\alpha$ is any  constant satisfying $\alpha>\frac{a^2}{\frac{1}{4}-a^2}$.

	Choose $\phi=\varphi_j v$, where $\varphi_j$ are the nonnegative cut-off functions   satisfying  that  $\varphi_j$ is $1$ on $B_j(0)$,  $|\nabla\varphi_j|\leq 1$ on $B_{j+1}(0)\setminus B_j(0)$, and $\varphi_j=0$ on $\Sigma\setminus B_{j+1}(0)$.  Substitute $\phi$ in (\ref{eq-bottom-2-1}):
	\begin{align}\label{eq-bottom-3-a}
	\lambda_1\int_{\Sigma}\varphi_j^2v^2e^{\frac{|x|^2}{4}}
	&\leq\int_{\Sigma}|\nabla (\varphi_jv)|^2e^{\frac{|x|^2}{4}}.
	\end{align}
	Note
	\begin{align}\label{eq-bottom-4-a}
	\int_{\Sigma}|\nabla (\varphi_jv)|^2e^{\frac{|x|^2}{4}}&\leq 2\int_{\Sigma}\varphi_j^2|\nabla v|^2e^{\frac{|x|^2}{4}}+2\int_{\Sigma}|\nabla\varphi_j|^2v^2e^{\frac{|x|^2}{4}},
	\end{align}
	and $v\in W^{1,2}(\Sigma,e^{\frac{|x|^2}{4}}d\sigma)$.  Letting $j\rightarrow \infty$ in  (\ref{eq-bottom-4-a}) and using the monotone convergence theorem, 
	\begin{align*}
	\int_{\Sigma}|\nabla (\varphi_jv)|^2e^{\frac{|x|^2}{4}}&\rightarrow \int_{\Sigma}|\nabla v|^2e^{\frac{|x|^2}{4}}.
	\end{align*}
	 Besides, since
	\begin{align}\label{nabla-v}
\int_{\Sigma}\varphi_jv(-\mathcal{L}v)e^{\frac{|x|^2}{4}}
 &=\int_{\Sigma}\langle \nabla v, \nabla(\varphi_jv)\rangle e^{\frac{|x|^2}{4}} \\
&\leq\frac12\int_{\Sigma}|\nabla v|^2 e^{\frac{|x|^2}{4}}+\frac12\int_{\Sigma}  |\nabla (\varphi_jv|^2  e^{\frac{|x|^2}{4}},\nonumber
\end{align}
letting $j\rightarrow \infty$ in  \eqref{nabla-v}, Lemma \ref{spn9} and the monotone convergence theorem yield
\begin{align*}
\int_{\Sigma}v(-\mathcal{L}v)e^{\frac{|x|^2}{4}}
 &=\int_{\Sigma}|\nabla v| ^2e^{\frac{|x|^2}{4}}.
 \end{align*}

	Then let $j\rightarrow \infty$ in  (\ref{eq-bottom-3-a}) and use the monotone convergence theorem again. We have
	\begin{align}\nonumber
	\lambda_1\int_{\Sigma}v^2e^{\frac{|x|^2}{4}}
	&\leq \int_{\Sigma}|\nabla v|^2e^{\frac{|x|^2}{4}}\\ \nonumber
	&=\int_{\Sigma}v(-\mathcal{L}v)e^{\frac{|x|^2}{4}}\\
	&=\int_{\Sigma}\left( \frac{(\alpha+1) n}{4}+\frac{\alpha+1}{2} H^2-\frac{(\alpha+1)(\alpha-1)}{16}|x^T|^2\right) v^2e^{\frac{|x|^2}{4}}.\label{eq-bottom-6-a}
	\end{align}
	Therefore
	\begin{align}\label{3-eq-thm2-1-c-1}
 	\lambda_1
 	&\leq \frac{(\alpha+1) n}{4}+\frac{(\alpha+1)\int_{\Sigma}\left( \frac12H^2-\frac{(\alpha-1)}{16}|x^T|^2\right) e^{-\frac{\alpha}{4}|x|^2}}{\int_{\Sigma}e^{-\frac{\alpha}{4}|x|^2}}.
 	\end{align}
 	
 	\end{proof}

	Now we prove Theorem  \ref{thm-3-hyper}, which is a corollary of Theorem  \ref{thm-3-hyper-1}.

	\begin{proof}[Proof of Theorem \ref{thm-3-hyper}]
	 Noting the assumption  $a<\frac{1}{2\sqrt{2}}$, we can  take $\alpha=1$  in Theorem \ref{thm-3-hyper-1}. Then 
	
	\begin{align}\label{spn}
	\lambda_1
	&\leq \frac{n}{2}+\frac{\int_{\Sigma}H^2e^{-\frac{|x|^2}{4}}}{\int_{\Sigma}e^{-\frac{|x|^2}{4}}}.
	\end{align}
	
	If the equality  in \eqref{spn} holds, \eqref{eq-bottom-6-a}  with $\alpha=1$ becomes the equality and hence  $v=e^{-\frac{1}{4}|x|^2}\in W^{1,2}(\Sigma,e^{\frac{|x|^2}{4}}d\sigma)$ is the first eigenfunction of the operator $\mathscr{L}$ satisfying
	\begin{align}\label{eq-v}
	\mathscr{L}v+\lambda_1v=0. 
	\end{align}
	Taking $\alpha=1$ in \eqref {espn1}, we have 
	\begin{align}\label{espn1-1}
\mathscr{L}v+\left( \frac{n}{2}+ H^2\right)v=0\hspace{0.2cm}\mbox{on}\hspace{0.2cm}\Sigma. 
\end{align}
	 \eqref{eq-v} and \eqref{espn1-1} imply $H=constant$.\\
	By
	\begin{align*}
	\mathscr{L}H+(|A|^2+\frac{1}{2})H=0,
	\end{align*}
	we conclude that $H\equiv 0$. Thus $\Sigma$ is a hyperplane passing through the origin.
\end{proof}

\section{Self-expander hypersurfaces with constant scalar curvature}\label{scalar}

 In this section we prove  Theorem \ref{scc6},   which characterizes the complete  self-expander surfaces  immersed in $\mathbb{R}^3$ with constant scalar curvature. We also prove Theorem \ref{sfc} which characterizes the complete  self-expander surfaces properly immersed  in $\mathbb{R}^3$ with the second fundamental form constant in norm and nonpositive scalar curvature. \\
 	
 	 In order to  prove Theorem \ref{scc6} we need the following result on the complete self-expander hypersurfaces immersed in $\mathbb{R}^{n+1}$ with nonnegative scalar curvature.

 	\begin{prop}\label{pmf}
 		Let $\Sigma$ be a complete  immersed self-expander hypersurface  in $\mathbb{R}^{n+1}$. Assume that $\Sigma$ is different from a  hyperplane and   has nonnegative scalar curvature. Then $\Sigma=\Gamma\times\mathbb{R}^{n-1}$, where $\Gamma$ is a complete non-trivial self-expander  curve  immersed in $\mathbb{R}^2$, if and only if the scalar curvature attains a local minimum  on the open set $\{x\in\Sigma; H(x)\neq0\}$.
 	\end{prop}
 	\begin{proof}
 		Since $Scal_{\Sigma}=H^2-|A|^2$ it follows that 
 		\begin{align}
 			\nabla Scal_{\Sigma}=2H \nabla H-\nabla|A|^2.
 		\end{align}
 		Therefore
 		\begin{align}
 			 4H^2|\nabla H|^2=\left\langle\nabla Scal_{\Sigma},\nabla Scal_{\Sigma}+2\nabla |A|^2 \right\rangle +4|A|^2|\nabla |A||^2.
 		\end{align}
 		This together with \eqref{esc} imply that on the set  $\{x\in\Sigma; H(x)\neq0\}$  the following holds 
 		
 		\begin{align}\nonumber
 			&\Delta Scal_{\Sigma}+\left\langle\nabla Scal_{\Sigma}, \frac{x}{2}-\frac{\nabla Scal_\Sigma+2\nabla |A|^2}{2H^2} \right\rangle\\\label{smp}
 			&=-Scal_{\Sigma}(2|A|^2+1)+2\frac{|A|^2}{H^2}|\nabla |A||^2-2|\nabla A|^2\\\nonumber
 			&\leq -Scal_{\Sigma}(2|A|^2+1)+2|\nabla |A||^2-2|\nabla A|^2\\\nonumber
 			&\leq 0.
 		\end{align} 
 	In the above we also use the  hypothesis $Scal_{\Sigma}\geq 0$ and the inequality $|\nabla A|^2-|\nabla |A||^2\geq0$.
 	
 		If $Scal_{\Sigma}$ attains a local minimum  on the set  $\{x\in\Sigma; H(x)\neq0\}$, the  maximum principle implies that there exists  an open set $U\subset\{x\in\Sigma; H(x)\neq0\}$ such that  $Scal_{\Sigma}$ is constant on $U$.  Further \eqref{smp} implies that $Scal_\Sigma=0$ on $U$. This implies that $\frac{|A|^2}{H^2}$ attains a local maximum on the open set $\{x\in\Sigma; H(x)\neq0\}$.  Smoczyk (\cite[Theorem 5.1]{smoczyk2020self}) proved that for a self-expander hypersurface immersed in $\mathbb{R}^{n+1}$ different from a linear subspace, it is of the
 		form $\Gamma\times\mathbb{R}^{n-1}$, where $\Gamma$ is a nontrivial self-expander curve in $\mathbb{R}^2$, if and only if, the function $\frac{|A|^2}{H^2}$ attains a local maximum on the open set $\{x\in\Sigma; H(x)\neq0\}$. Hence we conclude that $\Sigma=\Gamma\times\mathbb{R}^{n-1}$ where  $\Gamma$ is a nontrivial self-expander
 		curve in $\mathbb{R}^2$. This  completes the proof.
 	\end{proof}
 
 Proposition \ref{pmf} have the following consequence.
	\begin{theorem}\label{scc}
			Let $\Sigma$ be a complete immersed self-expander hypersurface in $\mathbb{R}^{n+1}$ with  nonnegative constant scalar curvature. Then $\Sigma=\Gamma\times\mathbb{R}^{n-1}$ with the product metric, where $\Gamma$ is a complete self-expander curve immersed in $\mathbb{R}^2$.
	\end{theorem}
\begin{proof}
	We consider two cases:\\
	(i) Case of $H\equiv0$. Obviously, $\Sigma$ is a hyperplane through the origin.\\
	(ii) Case of $H\neq0$.  By Proposition \ref{pmf} we conclude that $\Sigma=\Gamma\times\mathbb{R}^{n-1}$ where  $\Gamma$ is a nontrivial self-expander
	curve in $\mathbb{R}^2$. Combining these two cases completes the proof.
\end{proof}

 Now we prove Theorem \ref{scc6}. 
\begin{proof}[Proof of Theorem \ref{scc6}]
The classical Hilbert's theorem says that there is no complete surface immersed in $\mathbb{R}^3$ with negative constant scalar curvature. Hence  Theorem \ref{scc} implies Theorem \ref{scc6}.

In the case that $\Sigma$ is properly immersed, we may give an alternative direct proof without using Hilbert's theorem as follows.

First, we prove  that $Scal_\Sigma\geq0$.  Since $Scal_\Sigma=H^2-|A|^2$ is constant, we have
	\begin{align}\label{scc2}
		0=\nabla Scal_\Sigma=2H\nabla H-\nabla |A|^2
	\end{align}
	and, from  \eqref{esc},
	\begin{align}\label{scc3}
		Scal_\Sigma(2|A|^2+1)=2|\nabla H|^2-2|\nabla A|^2.
	\end{align}
	
	 Since $\Sigma$ is  properly immersed, there exists $p\in\Sigma$ which minimizes $|x|$. 
	At  the point $p$, we have 
	\begin{align}\label{scc1}
		\nabla H(p)=0.
	\end{align}  
	Choose a local orthonormal frame $\{e_1,e_2\}$ such that the coefficients of the second fundamental form are $h_{ij}(p)=\lambda_i\delta_{ij}$, for $i,j=1,2$.   By the definition,
	
	\begin{align*}
		|\nabla H|^2=(h_{111}+h_{221})^2+(h_{112}+h_{222})^2.
	\end{align*}
	This together with \eqref{scc1} implies
	\begin{align}\label{scc4}
		h_{111}=-h_{221}\hspace{1cm}\mbox{and }\hspace{1cm}	h_{222}=-h_{112}.
	\end{align}
	By \eqref{scc2} and \eqref{scc1},  we have $\nabla |A|^2=0$,  that is
	\begin{align*}
		h_{11}h_{111}+h_{22}h_{221}=h_{11}h_{112}+h_{22}h_{222}=0.
	\end{align*}
	Combining this with \eqref{scc4} yields
	\begin{align*}
		(h_{11}-h_{22})h_{111}=(h_{11}-h_{22})h_{222}=0.
	\end{align*}
	If $h_{11}=h_{22}$, we conclude that $Scal_\Sigma=H^2-|A|^2=|A|^2\geq0$.\\
	If $h_{111}=h_{222}=0$, then 
	\begin{align*}
		|\nabla A|^2=h^2_{111}+h^2_{222}+3h^2_{112}+3h^2_{221}=0. 
	\end{align*} 
 By \eqref{scc3}, we conclude that $Scal_{\Sigma}=0$  at $p$. \\
	Therefore,  we have that the constant  $Scal_\Sigma\geq0$.	 By Theorem \ref{scc}, we conclude that   $\Sigma=\Gamma\times\mathbb{R}$  with the product metric, where $\Gamma$ is a complete self-expander curve   immersed in $\mathbb{R}^2$.
\end{proof}
	Using an idea similar to the proof of the Theorem \ref{scc6}, we prove  Theorem \ref{sfc}.

	\begin{proof}[Proof of Theorem \ref{sfc}]
		Since $\Sigma$ is properly immersed, there exists $p\in\Sigma$ which minimizes $|x|$. At the point $p$, we have 
		\begin{align}\label{psa}
			\nabla H(p)=0.
		\end{align}
		Choose a local orthonormal frame $\{e_1,e_2\}$ such that the coefficients of the second fundamental form are $h_{ij}(p)=\lambda_i\delta_{ij}$, for $i,j=1,2$. By the  definition

		\begin{align*}
			|\nabla H|^2=(h_{111}+h_{221})^2+(h_{112}+h_{222})^2.
		\end{align*}
		This together with \eqref{psa} implies 
		\begin{align}\label{psa1}
			h_{111}=-h_{221}\hspace{1cm}\mbox{and }\hspace{1cm}	h_{222}=-h_{112}.
		\end{align}
		On the other hand, since $|A|$ is constant, we have 
		\begin{align}\label{psa3}
			h_{11}h_{111}+h_{22}h_{221}=h_{11}h_{112}+h_{22}h_{222}=0
		\end{align}
		and, by \eqref{Acon},
		\begin{align}\label{psa2}
			|\nabla A|^2=|A|^2(|A|^2+\frac12).
		\end{align}
		Combining \eqref{psa3} with \eqref{psa1} yields
		\begin{align*}
			(h_{11}-h_{22})h_{111}=(h_{11}-h_{22})h_{222}=0.
		\end{align*}
		If $h_{11}=h_{22}$, then $Scal_\Sigma=H^2-|A|^2=|A|^2\geq0$. This together with  hypothesis  $Scal_\Sigma\leq0$ implies that  $|A|=0$  at $p$. Since $|A|$ is constant, $|A|= 0$ on $\Sigma$.\\
		If $h_{111}=h_{222}=0$, then   
		\begin{align*}
			|\nabla A|^2=h^2_{111}+h^2_{222}+3h^2_{112}+3h^2_{221}=0.
		\end{align*} 
		By \eqref{psa2}, we conclude that $|A|=0$  at $p$. Therefore,  $|A|=0$ on $\Sigma$. Thus $\Sigma$ is a plane passing through the origin. 
\end{proof}

\section{Properties of stability operator $L$}\label{section-L}

 In this section, we discuss the  $L$-stability of self-expanders and estimate the bottom  spectrum of the $L$-stability operator  $ L=\mathcal{L}+|A|^2-\frac{1}{2}$.

\begin{proof}[Proof of Theorem \ref{staab}]
	Let $v=e^{-\frac{|x|^2}{4}}$. From \eqref{espn1} it follows that
	\begin{align}\label{spn2}
	\mathscr{L}v+(\frac{n}{2}+H^2)v=0.
	\end{align}
	Denote $w=\log(v)$. From \eqref{spn2}, we have 
	\begin{align*}
	\mathscr{L}w+|\nabla w|^2=-\frac{n}{2}-H^2.
	\end{align*}
	For any $\psi \in C^{\infty}_0(\Sigma)$, 
	\begin{align*}
	\frac{n}{2}\int_{\Sigma}\psi^2e^{\frac{|x|^2}{4}}+\int_{\Sigma} H^2 \psi^2e^{\frac{|x|^2}{4}}&=-\int_{\Sigma}|\nabla w|^2\psi^2e^{\frac{|x|^2}{4}}-\int_{\Sigma}\psi^2(\mathscr{L}w)e^{\frac{|x|^2}{4}}\\
	&=-\int_{\Sigma}|\nabla w|^2\psi^2 e^{\frac{|x|^2}{4}}+2\int_{\Sigma}\psi\langle \nabla \psi,\nabla w\rangle e^{\frac{|x|^2}{4}}\\
	&\leq\int_{\Sigma}|\nabla \psi|^2e^{\frac{|x|^2}{4}}.
	\end{align*}
	Therefore
	
	\begin{align}\nonumber
	\int_{\Sigma}\left( |\nabla \psi|^2-(|A|^2-\frac{1}{2})\psi^2\right) e^{\frac{|x|^2}{4}}&\geq\int_{\Sigma}\left( H^2-|A|^2+\frac{n+1}{2}\right)\psi^2e^{\frac{|x|^2}{4}}\\\nonumber
	&=\int_{\Sigma}\left(Scal_{\Sigma}+\frac{n+1}{2}\right)\psi^2e^{\frac{|x|^2}{4}}\\\label{staab3}
	&\geq \left( \frac{n+1}{2}+\inf_{x\in \Sigma}Scal_{\Sigma}\right)\int_{\Sigma}\psi^2e^{\frac{|x|^2}{4}}. 
	\end{align}
	Hence 
	\begin{align}\label{spn3}
	\mu_1\geq \frac{n+1}{2}+\inf_{x\in \Sigma}Scal_{\Sigma}.
	\end{align}
	Now we  have assumption that 	 $\Sigma$ is proper, satisfies  $|H|(x)\leq a|x|+b$, $x\in\Sigma$, for some constants $0\leq a<\frac{1}{2\sqrt{2}}$ and $b>0$, and has constant scalar curvature
	 $Scal_{\Sigma}=\inf_{x\in \Sigma}Scal_{\Sigma}$. 
	 
	By Inequality \eqref{bee1} in Theorem \ref{staab2}, which will be proved later in this paper, we have
\begin{align}\label{spn5-1}
	\mu_1\leq\frac{n+1}{2}+\frac{\int_{\Sigma}Scal_{\Sigma} e^{-\frac{|x|^2}{4}}}{\int_{\Sigma}e^{-\frac{|x|^2}{4}}}=\frac{n+1}{2}+\inf_{x\in \Sigma}Scal_{\Sigma}.
	\end{align}

	Inequalities  \eqref{spn3} and  \eqref{spn5-1} induce $\mu_1=\frac{n+1}{2}+Scal_{\Sigma}$ and hence the equality in \eqref{spn3} holds

\end{proof}

 Theorem \ref{staab} induces Corollary \ref{cor-stab} directly. Now we prove  the following Theorem \ref{staab2-1}, whose proof is similar to  that of Theorem \ref{thm-3-hyper-1}.
\begin{theorem} \label{staab2-1} Let $\Sigma$ be  a complete properly immersed self-expander hypersurface in $\mathbb{R}^{n+1}$. Assume that its mean curvature $H$ satisfies  $|H|(x)\leq a|x|+b$, $x\in\Sigma$,  for some constants  $0\leq a<\frac{1}{2}$ and $b>0$.   Then the bottom $\mu_1$ of the spectrum of the $L$-stability operator  $L$ satisfies
	\begin{align}\label{spn4}
	\mu_1
	\leq& \frac{(\alpha+1) n+2}{4}+\frac{\int_{\Sigma}\left[Scal_{\Sigma}+(\alpha-1)\big( \frac12H^2-\frac{(\alpha+1)}{16}|x^T|^2\big)\right]  e^{-\frac{\alpha}{4}|x|^2}d\sigma}{\int_{\Sigma}e^{-\frac{\alpha}{4}|x|^2}d\sigma}
	\end{align}
	for all $\alpha>\frac{a^2}{\frac{1}{4}-a^2}$.
\end{theorem}

\begin{proof}
	$\mu_1$ may be finite or $-\infty$. If $\mu_1$ is $-\infty$, the inequality \eqref{spn4} holds. Suppose $\mu_1>-\infty$. From the variational characterization of $\mu_1$, we have that for any  $\psi\in C^{\infty}_0(\Sigma)$,
	\begin{align}\label{staa2}
	\mu_1\int_{\Sigma}\psi^2e^{\frac{|x|^2}{4}}\leq\int_{\Sigma}|\nabla\psi|^2e^{\frac{|x|^2}{4}}-\int_{\Sigma}(|A|^2-\frac{1}{2})\psi^2e^{\frac{|x|^2}{4}},
	\end{align}
	Let $v=e^{-\frac{\alpha+1}{8}|x|^2}$, where $\alpha$ is any  constant satisfying $\alpha>\frac{a^2}{\frac{1}{4}-a^2}$. Recall \eqref{espn1},  i.e.,
	\begin{align}\label{se7}
	\mathscr{L}v+\left( \frac{(\alpha+1) n}{4}+\frac{\alpha+1}{2} H^2-\frac{(\alpha+1) (\alpha-1)}{16}|x^T|^2\right)v=0\hspace{0.2cm}\mbox{on}\hspace{0.2cm}\Sigma. 
	\end{align}
	 Lemmas \ref{spn9} and \ref{spn10} state that $v\in W^{1,2}(\Sigma,e^{\frac{|x|^2}{4}}d\sigma)$ and $\int_{\Sigma}H^2e^{-\frac{\alpha}{4}|x|^2}<\infty$. \\
	Choose $\psi=\varphi_j v$, where $\varphi_j$ are the nonnegative cut-off functions   satisfying  that  $\varphi_j$ is $1$ on $B_j(0)$,  $|\nabla\varphi_j|\leq 1$ on $B_{j+1}(0)\setminus B_j(0)$, and $\varphi=0$ on $\Sigma\setminus B_{j+1}(0)$.  Substitute $\psi$ in \eqref{staa2}:
	\begin{align}\label{staa3}
	\mu_1\int_{\Sigma}\varphi_j^2v^2e^{\frac{|x|^2}{4}}\leq\int_{\Sigma}|\nabla(\varphi_j v)|^2e^{\frac{|x|^2}{4}}-\int_{\Sigma}(|A|^2-\frac{1}{2})\varphi_j^2v^2e^{\frac{|x|^2}{4}}.
	\end{align}
	Note  $\mu_1>-\infty$. Letting $j\rightarrow \infty$ in  \eqref{staa3} and using  the monotone convergence theorem, \eqref{staa3}  implies that  $\int_{\Sigma}|A|^2e^{-\frac{\alpha}{4}|x|^2}<\infty$ and  
	\begin{align}\label{staa4}
	\mu_1\int_{\Sigma}v^2e^{\frac{|x|^2}{4}}&\leq\int_{\Sigma}|\nabla v|^2e^{\frac{|x|^2}{4}}-\int_{\Sigma}(|A|^2-\frac{1}{2})v^2e^{\frac{|x|^2}{4}}\nonumber \\\nonumber
		=\int_{\Sigma}&v
		(-\mathcal{L}v)e^{\frac{|x|^2}{4}}-\int_{\Sigma}(|A|^2-\frac{1}{2})v^2e^{\frac{|x|^2}{4}}\\\nonumber
	=\int_{\Sigma}&\left( \frac{(\alpha+1) n}{4}+\frac{\alpha+1}{2} H^2-\frac{(\alpha+1)(\alpha-1)}{16}|x^T|^2\right) v^2e^{\frac{|x|^2}{4}}\\ \nonumber
	-&\int_{\Sigma}(|A|^2-\frac{1}{2})v^2e^{\frac{|x|^2}{4}}\\
	=\int_{\Sigma}&\left[\frac{(\alpha+1) n+2}{4}+Scal_{\Sigma}+(\alpha-1)\bigg(\frac{H^2}2-\frac{\alpha+1 }{16}|x^T|^2\bigg)\right] v^2e^{\frac{|x|^2}{4}}.
	\end{align}
	Therefore 
	\begin{align}
	\mu_1
	\leq& \frac{(\alpha+1) n+2}{4}+\frac{\int_{\Sigma}\left[Scal_{\Sigma}+(\alpha-1)\big( \frac12H^2-\frac{(\alpha+1)}{16}|x^T|^2\big)\right]  e^{-\frac{\alpha}{4}|x|^2}}{\int_{\Sigma}e^{-\frac{\alpha}{4}|x|^2}}.
	\end{align}
	
\end{proof}

Theorem \ref{staab2-1} implies Theorem \ref{staab2} as follows:
\begin{proof}[Proof of Theorem \ref{staab2}]  The proof is similar to that of Theorem \ref{thm-3-hyper}.
Since $a<\frac{1}{2\sqrt{2}}$, we can   take $\alpha=1$ in Theorem \ref{staab2-1}.  Then 
	\begin{align}\label{spn5}
	\mu_1\leq\frac{n+1}{2}+\frac{\int_{\Sigma}Scal_{\Sigma} e^{-\frac{|x|^2}{4}}}{\int_{\Sigma}e^{-\frac{|x|^2}{4}}}.
	\end{align}
	If the equality in \eqref{spn5} holds, \eqref{staa4} with $\alpha=1$ becomes the equality. Namely,
	\begin{align}\label{bottom}
	\mu_1\int_{\Sigma}v^2e^{\frac{|x|^2}{4}}=\int_{\Sigma}|\nabla v|^2e^{\frac{|x|^2}{4}}-\int_{\Sigma}(|A|^2-\frac{1}{2})v^2e^{\frac{|x|^2}{4}}.
	\end{align}
	
	Thus  $\mu_1$ is attained on the function $v=e^{-\frac14|x|^2}\in W^{1,2}(\Sigma,e^{\frac{|x|^2}{4}}d\sigma)$.   One property of the bottom of spectrum (that can be proved similar to \cite[Theorem 10.10]{G})  induces that $v=e^{-\frac14|x|^2}$ must satisfies the equation
	\begin{align}\label{eq-v-2}
	\mathscr{L}v+(|A|^2-\frac{1}{2})v+\mu_1v=0.
	\end{align}
	Taking $\alpha=1$ in \eqref{se7}, we have 
	\begin{align}\label{se7-1}
\mathscr{L}v+\left( \frac{n}{2}+ H^2\right)v=0\hspace{0.2cm}\mbox{on}\hspace{0.2cm}\Sigma. 
\end{align}
By\eqref{eq-v-2} and \eqref{se7-1}, it follows  that  
$$Scal_{\Sigma}=H^2-|A|^2=\mu_1-\frac{n+1}{2}=constant.$$
Conversely, if  $Scal_{\Sigma}=constant$, by \eqref{spn5} and Theorem \ref{staab} it follows that 
	\begin{align*}
		\mu_1=Scal_{\Sigma}+\frac{n+1}{2}.
\end{align*}
	\end{proof}
	
In order to prove Corollary \ref{corr1},  we need  the following known  fact (see Theorem 3.20 in \cite{GSSZ}). For the convenience of  the readers, we include its proof here.
\begin{lemma}\label{cur3}
	 Let  $\gamma$ be a complete  immersed self-expander curve in $\R^2$ whose geodesic curvature $H$ satisfies $H=-\frac{1}{2}\langle x, {\bf n}\rangle$. Then either $\gamma$ is a straight line through the origin or $H^2$ is positive and bounded.
\end{lemma}
\begin{proof} Without lost of generality, we assume that $\gamma$ is parametrized by its arc-length. 
	 Differentiating   $H=-\frac{1}{2}\langle x, {\bf n}\rangle$ and noting that $\nabla_{\gamma'}{\bf n}=H\gamma'$  yield
	\begin{align}\label{cur1}
	2H'=-H\langle x,\gamma'\rangle.
	\end{align}
	On the other  hand,  differentiating $|x|^2$, we get 
	\begin{align}\label{cur2}
	(|x|^2)'=2\langle \gamma',x\rangle.
	\end{align}
	Combining \eqref{cur1} and \eqref{cur2} it follows that 
	\begin{align*}
	e^{\frac{-|x|^2}{2}}\left( H^2e^{\frac{|x|^2}{2}}\right)' =2HH'+\frac{1}{2}H^2\left(|x|^2 \right)'=-H^2\langle x,\gamma'\rangle+H^2\langle x,\gamma'\rangle=0.
	\end{align*}
	Therefore $H^2=Ce^{-\frac{|x|^2}{2}}$, for some constant $C$. This completes the proof. 
\end{proof}

With Lemma \ref{cur3} we prove Corollary \ref{corr1}.

\begin{proof}[Proof of Corollary \ref{corr1}]  Since $\Gamma$ is a complete immersed self-expander in $\mathbb{R}^2$, it is properly embedded by Theorem 6.1 in \cite{Hall}.
	Since  $\Sigma=\Gamma\times\R^{n-1}$,  $Scal_{\Sigma}\equiv 0$.  By Lemma \ref{cur3}, the mean curvature of $\Sigma$ is bounded. Hence, Theorems \ref{staab} and \ref{staab2}  imply that
	$\mu_1=\frac{n+1}{2}$. 
\end{proof}

Theorems  \ref{staab2} and   \ref{scc}   imply
\begin{cor}\label{corr2}
	Let $\Sigma$ be a complete properly immersed self-expander in $\R^{n+1}$. Assume that $Scal_{\Sigma}\geq0$ and $|H|(x)\leq a|x|+b$, $x\in\Sigma$, for some constants $0\leq a<\frac{1}{2\sqrt{2}}$ and $b>0$. Then 
	\begin{align}\label{spn6}
		\mu_1\leq\frac{n+1}{2}+\frac{\int_{\Sigma}Scal_{\Sigma} e^{-\frac{|x|^2}{4}}d\sigma}{\int_{\Sigma}e^{-\frac{|x|^2}{4}}d\sigma}.
	\end{align}
	The equality holds if and only if $\Sigma=\Gamma\times\R^{n-1}$  with the product metric,  where $\Gamma$ is a  complete self-expander curve (properly) immersed in $\R^2$.
\end{cor}


\begin{bibdiv}
\begin{biblist}



		
\bib{AIC}{article}{
   author={Angenent, S.},
   author={Ilmanen, T.},
   author={Chopp, D. L.},
   title={A computed example of nonuniqueness of mean curvature flow in $\mathbb{R}^3$},
   journal={Comm. Partial Differential Equations},
   volume={20},
   date={1995},
   number={},
   pages={1937-1958},
   issn={},
   review={},
   doi={},
}

\bib{AM}{article}{
	author={Ancari, Saul},
	author={Miranda, Igor},
	title={Volume estimates and classification theorem for constant weighted mean curvature hypersurfaces},
	journal={The Journal of Geometric Analysis },
	pages={ https://doi.org/10.1007/s12220-020-00413-2},
	year={2020}
	
}

\bib{BW1}{article}{
	author={Bernstein, Jacob },
	author={Wang, Lu},
	title={The space of asymptotically conical self-expanders of mean curvature flow},
	journal={arXiv:1712.04366v2 [math.DG]},
	 volume={},
   date={},
   number={},
   pages={},
	
}		

\bib{BW2}{article}{
	author={Bernstein, Jacob },
	author={Wang, Lu},
	title={Smooth compactness for spaces of asymptotically conical self-expanders of mean
curvature flow},
    journal={International Mathematics Research Notices},
	 volume={},
   date={},
   year={2019},
   number={},
   pages={doi:10.1093/imrn/rnz087},
	
}		



\bib{CH}{article}{
	author={Cao, Huai-Dong },
	author={Li, Haizhong},
	title={A gap theorem for self-shrinkers of the mean curvature flow in arbitrary codimension},
	journal={Calculus of Variations and Partial Differential Equations },
	 volume={46},
   date={2013},
   number={3-4},
   pages={879-889},
	
}		





		
\bib{CMZ3}{article}{
   author={Cheng, Xu},
   author={Mejia, Tito},
   author={Zhou, Detang},
   title={Stability and compactness for complete $f$-minimal surfaces},
   journal={Trans. Amer. Math. Soc.},
   volume={367},
   date={2015},
   number={6},
   pages={4041--4059},
   issn={0002-9947},
   review={\MR{3324919}},
   doi={10.1090/S0002-9947-2015-06207-2},
}
		
\bib{CMZ}{article}{
   author={Cheng, Xu},
   author={Mejia, Tito},
   author={Zhou, Detang},
   title={Eigenvalue estimate and compactness for closed $f$-minimal
   surfaces},
   journal={Pacific J. Math.},
   volume={271},
   date={2014},
   number={2},
   pages={347--367},
   issn={0030-8730},
   review={\MR{3267533}},
   doi={10.2140/pjm.2014.271.347},
}

\bib{CVZ}{article}{
	title={Volume growth of complete submanifolds in gradient Ricci Solitons with bounded weighted mean curvature},
	author={Cheng, Xu},
	author={Vieira, Matheus},
	author={Zhou, Detang},
	journal={International Mathematics Research Notices},
	year={2019}
}




		
		

		
\bib{CZ}{article}{
   author={Cheng, Xu},
   author={Zhou, Detang},
   title={Volume estimate about shrinkers},
   journal={Proc. Amer. Math. Soc.},
   volume={141},
   date={2013},
   number={2},
   pages={687--696},
   issn={0002-9939},
   review={\MR{2996973}},
   doi={10.1090/S0002-9939-2012-11922-7},
}

\bib{cheng2018}{article}{
	title={Spectral properties and rigidity for self-expanding solutions of the mean curvature
	flows},
	author={Cheng, Xu},
	author={Zhou, Detang},
	journal={Mathematische Annalen},
	volume={371},
	number={1-2},
	pages={371--389},
	year={2018},
	publisher={Springer}
}


	








\bib{D}{article}{
   author={Ding, Qi},
   title={Minimal cones and self-expanding solutions for mean curvature
   flows},
   journal={Math. Ann.},
   volume={376},
   date={2020},
   number={1-2},
   pages={359--405},
   issn={0025-5831},
   review={\MR{4055164}},
   doi={10.1007/s00208-019-01941-1},
}



\bib{EH}{article}{
   author={Ecker, Klaus},
   author={Huisken, Gerhard},
   title={Mean curvature evolution of entire graphs},
   journal={Ann. of Math.   2nd Ser.},
   volume={130},
   date={1989},
   number={3},
   pages={453-471},
   issn={},
   review={},
   doi={},
}


\bib{FM}{article}{
   author={Fong, Frederick Tsz-Ho},
   author={McGrath, Peter},
   title={Rotational symmetry of asymptotically conical mean curvature flow
   self-expanders},
   journal={Comm. Anal. Geom.},
   volume={27},
   date={2019},
   number={3},
   pages={599--618},
   issn={1019-8385},
   review={\MR{4003004}},
   doi={10.4310/CAG.2019.v27.n3.a3},
}




\bib{G}{book}{
   author={Alexander Grigoryan },
   title={Heat Kernel and Analysis on Manifolds	},
   language={English},
   publisher={American Mathematical Soc.},
   date={2009},
   pages={},
   review={},
}

\bib{GSSZ}{article}{
   author={Groh, Konrad },
   author={Schwarz, Matthias},
   author={Smoczyk, Knut},
   author={Zehmisch, Kai},
   title={ Mean curvature flow of monotone Lagrangian submanifolds,},
   journal={Math. Z.},
   volume={257},
   date={2007},
   number={2},
   pages={295-327},
   issn={},
   review={},
   doi={},
}


\bib{Hall}{article}{
   author={Halldorsson, Hoeskuldur P.},
   title={Self-similar solutions to the curve shortening
flow},
   journal={Trans. Amer. Math. Soc.},
   volume={364},
   date={2012},
   number={10},
   pages={5285–5309},
   issn={},
   review={},
   doi={},
}












\bib{I}{article}{
   author={Ilmanen,T},
   title={Lectures on Mean Curvature Flow and Related Equations (Trieste Notes),},
   journal={},
   volume={},
   date={1995},
   number={},
   pages={},
   issn={},
   review={},
   doi={},
}


\bib{Is}{article}{
   author={Ishimura, Naoyuki},
   title={Curvature evolution of plane curves with prescribed
opening angle},
   journal={Bull. Austral. Math. Soc.},
   volume={52},
   date={1995},
   number={2},
   pages={287–296},
   issn={},
   review={},
   doi={},
}




\bib{RS}{book}{
   author={Reed, Michael},
   author={Simon, Barry},
   title={Methods of modern mathematical physics. IV. Analysis of operators},
   publisher={Academic Press [Harcourt Brace Jovanovich, Publishers], New
   York-London},
   date={1978},
   pages={xv+396},
   isbn={0-12-585004-2},
   review={\MR{0493421}},
}

\bib{S}{article}{
   author={Stavrou, Nikolaos},
   title={Selfsimilar solutions to the mean curvature flow},
   journal={J. reine angew. Math.},
   volume={499},
   date={1998},
   number={},
   pages={189--198},
   issn={},
   review={},
   doi={},
}

\bib{smoczyk2020self}{article}{
	title={Self-expanders of the mean curvature flow},
	author={Smoczyk, Knut},
	journal={arXiv preprint arXiv:2005.05803},
	year={2020}
}






\end{biblist}
\end{bibdiv}
\end{document}